\documentclass[a4,12pt]{article}

\usepackage{amsfonts, amsmath, amssymb, amsgen, amsthm, amscd,latexsym,mathrsfs}

\usepackage{color}
\usepackage[all]{xy}

\usepackage[top=3.2cm, left=2.2cm, bottom=2.2cm, right=2.2cm]{geometry}

\def\div{\mathop{\rm div}\nolimits}

\def\Diff{\mathop{\rm Diff}\nolimits}

\def\Id{\mathop{\rm Id}\nolimits}

\def\log{\mathop{\rm log}\nolimits}

\def\Hom{\mathop{\rm Hom}\nolimits}

\def\Cb{{\mathbb C}}

\def\Nb{{\mathbb N}}
\def\Rb{{\mathbb R}}

\def\Ac{{\cal A}}

\def\Fc{{\cal F}}

\def\Hc{{\cal H}}

\def\Jc{{\cal J}}
\def\Ic{{\cal I}}

\def\Uc{{\cal U}}
\def\Vc{{\cal V}}

\def\a{\alpha}

\def\d{\delta}
\def\D{\Delta}

\def\om{\omega}
\def\Om{\Omega}
\def\s{\sigma}

\def\t{\theta}
\def\z{\zeta}
\def\ve{\varepsilon}
\def\vp{\varphi}

\def\0b{\bf 0}

\def\ot{\otimes}

\def\ra{\rightarrow}

\def\rt{\triangleright}

\def\lt{\triangleleft}

\def\acl{\blacktriangleright\hspace{-4pt}\vartriangleleft }

\def\hb{\overset{\ra}{b}}

\def\hB{\overset{\ra}{B}}

\def\p{\partial}

\def\0D{\Delta^{(0)}}
\def\1D{\Delta^{(1)}}
\def\Db{\blacktriangledown}

\def\wg{\wedge}

\def\cop{{^{\rm cop}}}

\newcommand{\wbar}[1]{\overline{#1}}

\newcommand{\FD}{\mathfrak{D}}

\newcommand{\Fg}{\mathfrak{g}}

\newcommand{\Fn}{\mathfrak{n}}

\newcommand{\Fd}{\mathfrak{d}}

\newcommand{\Fs}{\mathfrak{s}}

\newcommand{\Ft}{\mathfrak{t}}

\def\projot{\widehat{\ot}_\pi}

\newtheorem{theorem}{Theorem}[section]
\newtheorem{remark}[theorem]{Remark}
\newtheorem{proposition}[theorem]{Proposition}
\newtheorem{lemma}[theorem]{Lemma}
\newtheorem{corollary}[theorem]{Corollary}

\newtheorem{definition}[theorem]{Definition}

\def\ni{\noindent}

\def\build#1_#2^#3{\mathrel{
\mathop{\kern 0pt#1}\limits_{#2}^{#3}}}
\newcommand{\ps}[1]{~\hspace{-4pt}_{^{(#1)}}}
\newcommand{\pr}[1]{~\hspace{-4pt}_{_{\{#1\}}}}
\newcommand{\ns}[1]{~\hspace{-4pt}_{_{{<#1>}}}}

\def\odots{\ot\cdots\ot}
\def\wdots{\wedge\dots\wedge}

\def\one{{\bf 1}}

\newcommand{\ie}{{\it i.e.\/}\ }

\def\a{\alpha}

\def\d{\delta}

\def\om{\omega}
\def\s{\sigma}
\def\t{\theta}
\def\ve{\varepsilon}

\def\vp{\varphi}

\def\z{\zeta}

\def\D{\Delta}

\def\Om{\Omega}

\def\dt{\left.\frac{d}{dt}\right|_{_{t=0}}}

\def\ot{\otimes}
\def\part{\partial}

\def\wdg{\wedge}

\def\ra{\rightarrow}

\def\text{\hbox}

\def\ot{\otimes}

\def\ra{\rightarrow}

\def\wdg{\wedge}

\def\Diff{\mathop{\rm Diff}\nolimits}

\def\Hom{\mathop{\rm Hom}\nolimits}

\def\Id{\mathop{\rm Id}\nolimits}
\def\exp{\mathop{\rm exp}\nolimits}

\def\lra{\longrightarrow}
\def\lla{\longleftarrow}

\def \vDD{\uparrow\hspace{-1pt}d}
\def\hD{\overset{\ra}{d}}

\def\build#1_#2^#3{\mathrel{
\mathop{\kern 0pt#1}\limits_{#2}^{#3}}}

\numberwithin{equation}{section}
\title{\bf Hopf-cyclic cohomology of the Connes-Moscovici Hopf algebras with infinite dimensional coefficients}
\author{ B. Rangipour\thanks{University of New Brunswick, Department of Mathematics and Statistics,
      Fredericton, NB, Canada,\, Email: bahram@unb.ca, \,\, fereshteh.yazdani@unb.ca}  \qquad
       S. S\"utl\"u\thanks{I\c{s}{\i}k University, Department of Mathematics, 34980, \c{S}ile, {\.I}stanbul, Turkey, Email: \,serkan.sutlu@isikun.edu.tr} \quad \quad F. Yazdani Aliabadi$^\ast$ }

\date{}

\begin{document}
\maketitle

\abstract{
We discuss a new strategy for the computation of the Hopf-cyclic cohomology of the Connes-Moscovici Hopf algebra $\Hc_n$. More precisely, we introduce a multiplicative structure on the Hopf-cyclic complex of $\Hc_n$, and we show that the van Est type characteristic homomorphism from the Hopf-cyclic complex of $\Hc_n$ to the Gelfand-Fuks cohomology of the Lie algebra $W_n$ of formal vector fields on $\Rb^n$ respects this multiplicative structure. We then illustrate the machinery for $n=1$.
}

\tableofcontents

\section{Introduction}

The van Est type isomorphism 
\begin{equation*}
\kappa_n: \bigoplus_{i\,\equiv\, \ast\,\,({\rm mod}\,2)} \, H^i_{GF}(W_n,\Cb) \lra HP^\ast(\Hc_n,\d,1)
\end{equation*}
of \cite[Thm. 11]{ConnMosc98}, see also \cite[(4.12)]{ConnMosc}, and its relative version
\begin{equation*}
\kappa_{n,SO(n)}: \bigoplus_{i\,\equiv\, \ast\,\,({\rm mod}\,2)} \, H^i_{GF}(W_n,SO(n),\Cb) \lra HP^\ast(\Hc_n, SO(n),\d,1)
\end{equation*}
between the Hopf-cyclic cohomology (with trivial coefficients) of the Connes-Moscovici Hopf algebra $\Hc_n$, and the Gelfand-Fuks cohomology of the infinite dimensional Lie algebra $W_n$ of formal vector fields over $\Rb^n$ allowed a link between the characteristic classes of foliations and the total index class of the hypoelliptic signature operator \cite{ConnMosc95}. This way, the scope the theory of characteristic classes was broadened even further, \cite{ConnMosc04}. As such, a considerable amount of research on the Hopf algebra $\Hc_n$, and the (periodic) Hopf-cyclic cohomology Hopf (co)module (co)algebras has been initiated.

The first explicit computations on the Hopf-cyclic cohomology of the Connes-Moscovici Hopf algebras has been carried out by \cite{ConnMosc98,MoscRang07} for $\Hc_1$, using the bicrossed-product structure on $\Hc_n$. Those results were then followed by \cite{RangSutl-IV} for $\Hc_2$, in the presence of a cup product construction with an equivariant extension of the Hopf-cyclic cohomology. Finally, using a van Est type characteristic homomorphism through the Bott complex \cite{Bott76} and the simplicial de Rham complex \cite{Dupo76} of Dupont, Moscovici showed in \cite{Mosc14} that the elements of the Vey basis for the Gelfand-Fuks cohomology of $W_n$ can be transferred to the Hopf-cyclic cohomology of $\Hc_n$.

We, on the other hand, introduce in the present paper a multiplicative structure on the Hopf-cyclic cohomology complex of $\Hc_n$ (and in the presence of a highly non-trivial coefficients), and show that our van Est type characteristic homomohism \cite{RangSutl-III} between the Gelfand-Fuks cohomology of $W_n$ and the Hopf-cyclic cohomology of $\Hc_n$ respects the multiplicative structures on its domain and range. Thus, we can move the characteristic classes to the Hopf-cyclic cohomology by transfering only the multiplicative generators, and thus obtain a (Vey) basis for the Hopf-cyclic cohomology of $\Hc_n$.

The Hopf algebra $\Hc_n$ is introduced in \cite{ConnMosc98}, for each $n\in \Nb$, as an organisational device in the computation of the index of the tranversally elliptic operators on foliations. By its very nature, $\Hc_n$ is a Hopf algebra of differential operators on the bundle $F^+(M)$ of orientation preserving frames on a flat $n$-manifold $M$, and it thus acts naturally on the cross-product algebra $\Ac_\Gamma:= C^\infty(F^+) \rtimes \Gamma$, for any pseudogroup $\Gamma$ of partial diffeomorphisms on $F^+$. The structure of $\Hc_n$ has been investigated extensively through \cite{FiguGracBond05,FiguGracBondVari05,HadfMaji07,MoscRang07,MoscRang09}. 

It was first observed in \cite{HadfMaji07} that $\Hc_1$ is a bicrossproduct Hopf algebra. Then in \cite{MoscRang07,MoscRang09} the authors showed, using its module algebra action on the algebra $\Ac_\Gamma:= C^\infty(F^+) \rtimes \Gamma$, that this is in fact the case for any $n \in \Nb$.

The domain of the van Est type map, Hopf-cyclic cohomology, is introduced in \cite{ConnMosc98} as a cyclic cohomology theory associated to a Hopf algebra and a pair of elements (called the modular pair in involution, or MPI in short) consisting of a grouplike element in the Hopf algebra, and a character of the Hopf algebra. The theory was then developed through \cite{HajaKhalRangSomm04-II, HajaKhalRangSomm04-I,Kayg05} as a cyclic cohomology theory associated to a (co)algebra, equipped with a Hopf algebra (co)action, and a particular (co)representation of that Hopf algebra as the space of coefficients (called stable-anti-Yetter-Drinfeld modules, or SAYD modules in short), so that \cite{ConnMosc98}'s $HP^\ast(\Hc_n,\d,1)$ is the (periodic) Hopf-cyclic cohomology with trivial coefficients.

It turned out that the bicrossproduct structure of $\Hc_n$ was not only helpful in understanding its Hopf algebra structure, but is was also crucial to compute its Hopf-cyclic cohomology. This point of view was taken in \cite{MoscRang09} to introduce a bicocyclic bicomplex computing the Hopf-cyclic cohomology (with trivial coefficients) of $\Hc_n$. 

On the other hand, nontrivial examples of SAYD modules over bicrossproduct Hopf algebras were developed through \cite{RangSutl-II,RangSutl-III,RangSutl}. More precisely, given a bicrossproduct Hopf algebra associated to a Lie algebra via semi-dualisation \cite{ Maji90,Majid-book}, a SAYD module was associated to any representation of the Lie algebra. In \cite{RangSutl-III}, a concrete 4-dimensional SAYD module over the Schwarzian quotient $\Hc_{{\rm 1S}}$ of $\Hc_1$ was constructed this way, and the Hopf-cyclic cohomology of $\Hc_{{\rm 1S}}$ with coefficients in this particular space were computed. Furthermore, it was also observed in \cite{RangSutl-III} that the Connes-Moscovici Hopf algebra $\Hc_n$ is an example of a semi-dualisation Hopf algebra associated to the Lie algebra $W_n$ of formal vector fields on $\Rb^n$, and since $W_n$ has no nontrivial finite dimensional representation, $\Hc_n$ does not admit any nontrivial finite dimensional SAYD module.

It was this last result that prompted us to think about the Hopf-cyclic cohomology of $\Hc_n$ with infinite dimensional coefficients. In fact, an example of an infinite dimensional SAYD module over a Hopf subalgebra of $\Hc_1$ was already introduced in \cite{Antal-thesis}. However, there appears to be no attempt in the literature regarding an explicit computation of the Hopf-cyclic cohomology of Connes-Moscovici Hopf algebras with infinite dimensional coefficients. 

Now the range of the van Est type homomorphism, the Gelfand-Fuks cohomology of the Lie algebra $W_n$ of formal vector fields on $\Rb^n$ was the target of a series of attempts \cite{GelfFuks69-III,GelfFuks69-I-II,GelfFuks70-II,GelfFuks70-III,GelfFuks70-IV}. It is known to be finite dimensional \cite{GelfFuks70-III,GelfFuks70-V}, and provides a universal source for all characteristic classes of foliations \cite{Bott76}. On the other hand, the cohomology of $W_n$ with nontrivial coefficients has been studied through \cite{GelfFeigFuks74,GelfFuks70,Losi70}, see also \cite{Fuks-book}. In the present paper we shall consider the cohomology of $W_n$ with coefficients in the space of formal differential forms.

In the case of the trivial coefficients, the van Est type characteristic map between the Gelfand-Fuks cohomology of $W_n$ and the Hopf-cyclic cohomology of $\Hc_n$ has also been considered in \cite{MoscRang11, MoscRang15} from the point of view of the integration of invariant forms over simplexes in the spaces of jets of diffeomorphims. In \cite{RangSutl-III,RangSutl}, however, the transfer of classes was achieved via a characteristic isomorphism in the opposite direction, \ie from the Hopf-cyclic cohomology to the Gelfand-Fuks cohomology via differentiation. Here we shall adopt this last point of view, introduce a multiplicative structure on the Hopf-cyclic cohomology (with coefficients) bicomplex of $\Hc_n$, and show that the characteristic homomorphism respects the multiplicative structures on its domain and the range.

The plan of the paper is as follows. In Section \ref{diff-forms} we consider the space $\Om_n^{\leq1}$ of formal differential 0-forms together with 1-forms on $\Rb^n$. We review the bicrossproduct structure of the Connes-Moscovici Hopf algebra $\Hc_n$, and then we illustrate the (induced) SAYD module structure of $\Om_n^{\leq1}$ over $\Hc_n$. Section \ref{Lie-cohom} is devoted to the Lie algebra cohomology, with coefficients. In particular, we recall the cohomology of a matched pair Lie algebra, as well as the cohomology of $W_n$ with coefficients in the space $\Om_n^{\leq1}$. In Section \ref{Hopf-cyclic} we recall the Hopf-cyclic cohomology, with coefficients, for Hopf algebras. More importantly, it is Section \ref{Hopf-cyclic} in which we introduce a multiplicative structure on the Hopf-cyclic bicomplex. Finally, we show in Section \ref{classes} that our characteristic isomorphism respects the multiplicative structures on the Hopf-cyclic complex of $\Hc_n$ and the Lie algebra cohomology complex of $W_n$. We illustrate the whole discussion in the case $n=1$. More explicitly, we transfer the generators of $H^\ast(W_1,\Om_1^{\leq1})$ to the Hopf-cyclic cohomology $HC^\ast(\Hc_1,\Om_{1\d}^{\leq1})$.

\section{The space of formal differential forms}\label{diff-forms}

\subsection{The Connes-Moscovici Hopf algebra $\Hc_n$}

We recall, in this section, the Connes-Moscovici Hopf algebra $\Hc_n$, and its bicrossed product structure from \cite{ConnMosc98,MoscRang09}. Referring the reader to \cite{Maji90,Majid-book,Sing70} for a quick review of the bicrossed product Hopf algebras, as well as the matched pairs of Lie groups and Lie algebras, 
we begin with the note that we are going to use the Sweedler's notation \cite{Sweed-book} for the coaction and the comultiplication.

\ni From the group decomposition point of view, the Connes-Moscovici Hopf algebra $\Hc_n$ is constructed by the Kac decomposition, \cite{Kac68}, of the group $\Diff(\Rb^n)$ of diffeomorphisms of $\Rb^n$. Accordingly, $\Diff(\Rb^n) = G\times N$, where $G \cong F\Rb^n \cong GL_n^{\rm aff}$ is the group of affine transformations, and 
\begin{equation*}
N = \{\phi \in \Diff(\Rb^n)\mid \phi(0)=0,\,\, \phi'(0)=\Id\}.
\end{equation*}
We thus have the Hopf algebra $\Uc:=U(g\ell_n^{\rm aff})$, where
\begin{align*}
& g\ell_n^{\rm aff} = \langle \{X_k, Y_i^j \mid 1 \leq i,j,k \leq n\}\rangle, \\ 
& [Y_i^j, X_k]=\d^j_kX_i, \quad [X_k,X_\ell]=0,\quad [Y_i^j,Y_p^q]=\d^j_pY_i^q - \d_i^qY^j_p
\end{align*}
is the Lie algebra of the group $GL_n^{\rm aff}:=\Rb^n\rtimes GL_n$, and the Hopf algebra $\Fc:=\Fc(N)$ of regular functions on $N$ generated by the functions given by
\begin{equation*}
\a^i_{jk_1\ldots k_r}(\psi) = \p_{k_r}\ldots \p_{k_1}\p_j(\psi^i(x))|_{x=0}, \qquad 1 \leq i,j,k_1,\ldots, k_r\leq n,\,\, \psi\in N,
\end{equation*}
or alternatively by the functions
\begin{equation*}
\eta^i_{jk\ell_1\ldots \ell_r}(\psi) = \p_{\ell_r}\ldots \p_{\ell_1}\left((\psi'(x)^{-1})^i_\nu\p_j\p_k\psi^\nu(x)\right)|_{x=0}.
\end{equation*}
The Hopf algebra $\Fc$ is a $\Uc$-module algebra by the action
\begin{equation*}
(Z\rt f)(\psi):=\left.\frac{d}{dt}\right|_{t=0}f(\psi\lt \exp(tZ)), \qquad f\in \Fc, \,\, Z\in g\ell_n^{\rm aff},
\end{equation*}
and $\Uc$ is a $\Fc$-comodule coalgebra by the coaction 
\begin{align}\label{coact-g-F}
\begin{split}
& \Db:g\ell_n^{\rm aff}\lra g\ell_n^{\rm aff}\ot \Fc, \\ 
& \Db(X_k) = X_k\ot 1 + Y_i^j\ot \eta^i_{jk}, \qquad \Db(Y_i^j)=Y_i^j\ot 1
\end{split}
\end{align}
which is extended to a coaction $\Db:\Uc\to \Uc\ot \Fc$. 

\begin{remark} 
{\rm
In view of the non-degenerate pairing \cite[(3.50)]{RangSutl-III}, see also \cite[Prop. 3]{ConnMosc98} or \cite[Prop. 3]{ConnKrei98}, $\Fc$ is isomorphic with the Hopf algebra $R(\Fn)$ of representative functions on $U(\Fn)$, where $\Fn$ is the Lie algebra of the group $N$, and the coaction \eqref{coact-g-F} dualizes the left $\Fn$-action on $g\ell_n^{\rm aff}$.
}
\end{remark}

\ni Finally, it follows from \cite[Prop. 2.14]{MoscRang09} that $(\Fc,\Uc)$ is a matched pair of Hopf algebras, and from \cite[Thm. 2.15]{MoscRang09} that ${\Hc_n}\cop \cong \Fc\acl \Uc$.

\ni To review the bicrossed product structure of $\Hc_n$, from the Lie algebra decomposition point of view, we consider the Lie algebra $W_n$ of formal vector fields ton $\Rb^n$. Elements of $W_n$ are expressed as $\sum_{i=1}^n f^i(x^1,\ldots, x^n)\p_i$, where $f^i(x^1,\ldots, x^n)$ is a formal power series in the indeterminates $x^1,\ldots, x^n$, for any $i=1,\ldots, n$. It is an infinite dimensional vector space
\begin{equation*}
W_n = \left\langle\{e_i:= \p_i,\,\,e^j_i:= x^j\p_i,\,\,e_i^{jk_1\ldots k_r}:= x^jx^{k_1}\ldots x^{k_r}\p_i \mid 1\leq i,j,k_1 \ldots k_r \leq n \}\right\rangle,
\end{equation*}
with the Lie bracket given by
\begin{align*}
& [e_i,e_j]=0, \qquad [e_k,e_i^j] = \d_k^je_i, \qquad [e_\ell,e_i^{jk_1 \ldots k_r}] =\d^j_\ell e_i^{k_1 \ldots k_r}, \\
& [e_i^j,e_p^{q\ell_1 \ldots \ell_r}] = \d_i^qe_p^{j\ell_1 \ldots \ell_r} + \sum_{m=1}^n\d_i^{\ell_m}e_p^{jq\ell_1 \ldots \widehat{\ell_m} \ldots \ell_r} - \d_p^je_i^{q\ell_1 \ldots \ell_r},\\
& [e_i^{jk_1 \ldots k_r}, e_p^{q\ell_1 \ldots \ell_s}] = \\
& \d_i^qe_p^{jk_1 \ldots k_r\ell_1 \ldots \ell_s} + \sum_{m=1}^s\d_i^{\ell_m}e_p^{jqk_1 \ldots k_r\ell_1 \ldots \widehat{\ell_m}\ldots \ell_s} - \d_p^je_i^{q\ell_1 \ldots \ell_s k_1 \ldots k_r} \\
&- \sum_{m=1}^r\d_p^{k_m}e_i^{jq\ell_1 \ldots \ell_sk_1 \ldots \widehat{k_m}\ldots k_r}.
\end{align*}
Setting 
\begin{equation*}
\Fs:=\left\langle\{e_i:= \p_i,\,\,e^j_i:= x^j\p_i \mid 1\leq i,j \leq n \}\right\rangle \cong g\ell_n^{\rm aff},
\end{equation*}
and 
\begin{equation*}
\Fn := \left\langle\{e_i^{jk_1\ldots k_r}:= x^jx^{k_1}\ldots x^{k_r}\p_i \mid 1\leq i,j,k_1 \ldots k_r \leq n \}\right\rangle,
\end{equation*}
we obtain at once the matched pair decomposition $W_n=\Fs\bowtie \Fn$. The mutual actions are, via \cite[Prop. 8.3.2]{Majid-book},
\begin{equation*}
e_i^{jk_1\ldots k_r} \rt e_\ell = \begin{cases}
-\d^j_\ell e_i^{k_1}, &  \text{ if } r=1, \\
0, & \text{ if } r\geq 2,
\end{cases} \qquad e_i^{jk_1\ldots k_r} \rt e_p^q = 0,
\end{equation*}
and 
\begin{align*}
& e_i^{jk_1\ldots k_r} \lt e_\ell = \begin{cases}
0, & \text{ if } r=1, \\
-\d^j_\ell e_i^{k_1\ldots k_r}, & \text{ if } r\geq 2,
\end{cases} \\
& e_i^{jk_1\ldots k_r} \lt e_p^q = \d_i^qe_p^{jk_1 \ldots k_r} -\d_p^je_i^{qk_1 \ldots k_r} - \sum_{m=1}^n\d_p^{k_m}e_i^{jqk_1 \ldots \widehat{k_m} \ldots k_r}.
\end{align*}

\ni We next recall the concept of a Lie-Hopf algebra, \cite{RangSutl} and see also \cite{RangSutl-V}.

\begin{definition}
Let a Lie algebra $\Fg$ act on a commutative Hopf algebra $\Fc$ by derivations, and $\Fc$ coacts on $\Fg$. Then $\Fc$ is said to be a $\Fg$-Hopf algebra if
\begin{enumerate}
\item the coaction $\Db:\Fg\to \Fg\ot \Fc$ of $\Fc$ on $\Fg$ is a map of Lie algebras, where the bracket on $\Fg\ot\Fc$ is given by
\begin{equation}\label{top-bracket}
[X\ot f, Y\ot g]:= [X,Y]\ot fg + Y\ot \ve(f)X\rt g- X\ot\ve(g)Y\rt Y,
\end{equation}
for any $X,Y\in \Fg$, and any $f,g \in \Fc$,
\item the comultiplication and counit of $\Fc$ are $\Fg$-linear, \ie $\D(X\rt f) = X \bullet \D(f)$, and $\ve(X\rt f) = 0$, where the latter action is given by
\begin{align}\label{bullet}
\begin{split}
&X\bullet (f^1\odots f^q) := \\
&\hspace{1cm} X\ps{1}\ns{0}\rt f^1 \ot X\ps{1}\ns{1}X\ps{2}\ns{0}\rt f^2 \ot\cdots\\
&\cdots\ot X\ps{1}\ns{q-1}\ldots X\ps{q-1}\ns{1}X\ps{q}\rt f^q,
\end{split}
\end{align}
for any $X \in \Fg$, and any $f^1,\ldots, f^q\in \Fc$.
\end{enumerate}
\end{definition}

\ni The proof of the following proposition is similar to that of \cite[Prop. 2.10]{RangSutl-I-arxiv}, and hence is omitted.

\begin{proposition}\label{s-induced}
The commutative Hopf algebra $\Fc=\Fc(N)$ is an $\Fs$-Hopf algebra.
\end{proposition}

\ni As a result, it follows from \cite[Thm. 2.6]{RangSutl-I-arxiv}, see also \cite[Thm. 2.14]{RangSutl-V}, that $(\Fc(N),U(\Fs))$ is a matched pair of Hopf-algebras, and the bicrossed product Hopf algebra $\Fc(N)\acl U(\Fs) = \Fc(N)\acl U(g\ell_n^{\rm aff})$ is isomorphic (as Hopf algebras) with $\Hc_n\cop$.

\subsection{SAYD structure over $\Hc_n$}

It was observed in \cite{RangSutl-III} that the only finite dimensional AYD module over the Connes-Moscovici Hopf algebra $\Hc_n$ is the trivial one, $\Cb_\d$. On the other hand, the dual of the space of formal exterior differential 1-forms was considered in \cite{Antal-thesis} as an infinite dimensional nontrivial example, over a Hopf subalgebra of $\Hc_1$. In this section we study the space of formal differential $\leq 1$-forms as an infinite dimensional coefficient space for the Hopf-cyclic cohomology of the Hopf algebra $\Hc_n$.

\ni Let us recall from \cite{HajaKhalRangSomm04-II} that a vector space
 $V$ is called a right-left
stable-anti-Yetter-Drinfeld (SAYD) module over $H$ if it
is a right
 $H$-module, a left $H$-comodule, and
\begin{equation}\label{aux-SAYD-condition}
\nabla(v\cdot h)= S(h\ps{3})v\ns{-1}h\ps{1}\ot v\ns{0}\cdot h\ps{2},\qquad  v\ns{0}\cdot v\ns{-1}=v,
\end{equation}
for any $v\in V$ and any $h\in H$.

\ni Adopting the notation of \cite{Fuks-book}, we let $\Om_n^q$ to denote the space of formal exterior differential $q$-forms on $\Rb^n$. In particular, $\Om_n^0$ is the space of formal power series in $x^1,\ldots, x^n$, and
\begin{equation*}
\Om_n^1 := \{f_idx^i \mid f_i \text{ is a formal power series of } x^1,\ldots, x^n\}
\end{equation*}
is the space of formal differential 1-forms. The space $\Om_n^1$ is an infinite dimensional vector space, and it has a natural $W_n$-module structure, \cite[Subsect. 2.2.4]{Fuks-book}. We shall, in particular, consider the space $\Om_n^{\leq1}:=\Om_n^0\oplus\Om_n^1$, which is naturally a $U(\Fn)$-module. Transposing the action $U(\Fn) \ot \Om_n^{\leq1} \to \Om_n^{\leq1}$, we obtain
\begin{equation*}
{\Om_n^{\leq1}}^\ast \lra \left(U(\Fn) \ot \Om_n^{\leq1}\right)^\ast.
\end{equation*}
which factors through the embedding
\begin{equation*}
U(\Fn)^\ast \ot {\Om_n^{\leq1}}^\ast \lra \left(U(\Fn) \ot \Om_n^{\leq1}\right)^\ast,
\end{equation*}
see for instance \cite[Sect. 5.3]{Antal-thesis}. 

\ni It then follows from ${\Om_n^\lambda}^\ast = \Om_n^{1-\lambda}$, see \cite{OvsiRoge98}, and the nondegenerate pairing between $U(\Fn)$ and $\Fc(N)$ that we have a (left, and then using the antipode) right coaction
\begin{equation*}
\Db:\Om_n^{\leq1} \lra \Om_n^{\leq1} \ot \Fc(N), \qquad \om\mapsto \om\ns{0}\ot \om\ns{1},
\end{equation*}
so that, given any $v \in U(\Fn)$, $\om\ns{0}\,\om\ns{1}(v) = v \rt \om$, see also \cite[Eqn. (5.42)]{Antal-thesis}.

\begin{remark}
{\rm
We remark that in the expense of passing to the topological vector spaces (in the sense of \cite{BonnFlatGersPinc94,BonnSter05}) and their tensor product (for which we refer the reader to \cite{Schaefer-book,Treves-book}) we may always dualize the above left action to a right coaction.
}
\end{remark}

\ni Following \cite{RangSutl,RangSutl-V}, we shall observe that $\Om_n^{\leq 1}$ is an induced SAYD module over the Hopf algebra $\Fc(N) \acl U(\Fs)$. 
 We therefore recall its definition.
\begin{definition}
Let $\Fg$ be a Lie algebra, and $\Fc$ a $\Fg$-Hopf algebra. Let also $M$ be a (left) $\Fg$-module, and a right $\Fc$-comodule via $\Db:M \to M \ot\Fc$. We then call $M$ an induced $(\Fg,\Fc)$-module if 
\begin{equation}\label{induced-module}
\Db(X\cdot m) = X \bullet \Db(m)
\end{equation} 
for any $X \in \Fg$, any $m \in M$, and any $f \in \Fc$.
\end{definition}

\begin{lemma}
The space $\Om_n^{\leq 1}$ is an induced $(\Fs,\Fc(N))$-module.
\end{lemma}

\begin{proof}
We first recall that $\Fc(N)$ being a $\Fs$-Hopf algebra was observed already in Proposition \ref{s-induced}. We are thus left to show \eqref{induced-module}. To this end we observe that
\begin{align*}
& \langle X \bullet (\om\ns{0} \ot \om\ns{1}),\, v\rangle = \\
& \langle X^{\pr{0}} \cdot \om\ns{0} \ot X^{\pr{1}}\om\ns{1},\,v\rangle + \langle \om\ns{0}\ot X\rt \om\ns{1},\,v\rangle = \\
& (v\ps{1}\rt X) \cdot (v\ps{2}\rt \om) + (X\lt v) \cdot \om = v \cdot (X \cdot \om) = \langle \Db(X \cdot \om),\, v\rangle
\end{align*}
for any $v\in U(\Fn)$, where the third equality follows from \cite[(3.35)]{RangSutl-III}. The claim thus follows from the non-degeneracy of the pairing between $U(\Fn)$ and $\Fc(N)$.
\end{proof}

\ni As a result of \cite[Prop. 3.4]{RangSutl}, $\Om_n^{\leq 1}$ is a left / right YD-module over the bicrossed product Hopf algebra $\Fc(N) \acl U(\Fs)$ via the action
\begin{equation*}
\Fc(N) \acl U(\Fs) \ot \Om_n^{\leq 1}  \lra \Om_n^{\leq 1}, \qquad (f\acl u) \cdot \om := \ve(f)u \cdot \om,
\end{equation*}
and the coaction
\begin{equation*}
\Om_n^{\leq 1} \lra \Om_n^{\leq 1} \ot \Fc(N) \acl U(\Fs), \qquad \Db(\om) := \om\ns{0} \ot (\om\ns{1}\acl 1).
\end{equation*}
Finally, since $(\d,1)$ is a MPI on the Hopf algebra $\Fc(N) \acl U(\Fs)$, see for instance \cite{ConnMosc98,MoscRang09} or \cite[Thm. 3.2]{RangSutl}, we conclude that $\Om_{n\d}^{\leq 1}:={}^1 \Cb_\d \ot \Om_n^{\leq1}$ is a right / left SAYD module over $\Fc(N) \acl U(\Fs)$ via the action
\begin{equation*}
\Om_{n\d}^{\leq 1} \ot \Fc(N) \acl U(\Fs) \lra \Om_{n\d}^{\leq 1}, \qquad \om \cdot (f\acl u) := \ve(f)\d(u\ps{1})S(u\ps{2}) \cdot \om,
\end{equation*}
and the coaction
\begin{equation*}
\Om_{n\d}^{\leq 1} \lra \Fc(N) \acl U(\Fs) \ot \Om_{n\d}^{\leq 1} , \qquad \Db(\om) := (S(\om\ns{1})\acl 1) \ot \om\ns{0}.
\end{equation*}

\section{Lie algebra cohomology $H^\ast(W_n,\Om_n^{\leq1})$}\label{Lie-cohom}

In this section we recall the Lie-algebra cohomology with coefficients. In particular, we shall discuss the cohomology of the infinite dimensional Lie algebra of formal vector fields, with coefficients in the space of formal differential 1-forms, \cite{GelfFuks70-II,Fuks-book}. 

\subsection{Lie algebra cohomology with coefficients}\label{Lie-cohom-coeff}

Let $\Fg$ be a Lie algebra, and $M$ a $\Fg$-module. Then the graded space
\begin{equation*}
C^\ast(\Fg,M)=\bigoplus_{k\geq 0}C^k(\Fg,M), \qquad C^k(\Fg,M):=\Hom(\wg^k\Fg,M)
\end{equation*}
is a differential graded space via
\begin{align*}
& d_{\rm CE}:C^k(\Fg,M) \lra C^{k+1}(\Fg,M), \\
& d_{\rm CE}c(\xi_1,\ldots, \xi_{k+1}) := \sum_{1\leq r < s \leq k+1} (-1)^{r+s-1}c([\xi_r,\xi_s],\xi_1,\ldots,\widehat{\xi_r},\ldots, \widehat{\xi_s},\ldots,\xi_{k+1}) \\
& \hspace{3.4cm}+ \sum_{t=1}^{k+1}(-1)^{t}\xi_t\cdot c(\xi_1,\ldots,\widehat{\xi_t},\ldots,\xi_{k+1}),
\end{align*}
or alternatively via
\begin{align*}
& d_{\rm CE}(m) = m\cdot X_i \ot \t^i, \\
& d_{\rm CE}(m\ot \eta) = m\cdot X_i \ot \t^i\wg \eta + m\ot d_{\rm DR}(\eta),
\end{align*}
where $d_{\rm DR}:\wg^p \Fg^\ast \to \wg^{p+1} \Fg^\ast$ is the deRham coboundary (which is a derivation of order 1) given by 
\begin{equation*}
d_{\rm DR}(\t^k) = \frac{1}{2}C^k_{ij}\t^i\wg\t^j.
\end{equation*}
The homology of the differential graded space $(C^\ast(\Fg,M),d_{\rm CE})$ is called the Lie algebra cohomology of $\Fg$, with coefficients in $M$, and is denoted by $H^\ast(\Fg,M)$.


\ni We shall recall from \cite{HochSerr53} the multiplicative structure on the Lie algebra cohomology. Let $M$, $M'$, and $P$ be $\Fg$-modules. Then $M$ and $M'$ are said to be paired to $P$ if there exists a bilinear mapping $M\times N \to P$, $(m,m')\mapsto m\cup m'$, such that
\begin{equation*}
\xi\cdot (m\cup m') := \xi\cdot m \cup m' + m\cup \xi\cdot m',
\end{equation*}
for any $m\in M$, any $m'\in M'$, and any $\xi \in \Fg$. Let also $S=\{s_1,\ldots,s_p\}$ be an ordered subset of integers in $\{1,2,\ldots,p+q\}$, and $T=\{t_1,t_2,\ldots,t_q\}$ be its ordered complement. For each $1\leq j\leq q$, let $S(j)$ denote the number of indices $i$ for which $s_i>t_j$, and let $\nu(S):=\sum_{j=1}^qS(j)$. Then,
\begin{equation}\label{Lie-multp}
(c \cup c')(\xi_1,\ldots,\xi_{p+q}) := \sum_S (-1)^{\nu(S)} c(\xi_{s_1},\ldots,\xi_{s_p})\cup c'(\xi_{t_1},\ldots,\xi_{t_q})
\end{equation}
defines an element $c\cup c' \in C^{p+q}(\Fg,P)$, called the cup product of $c\in C^p(\Fg,M)$ and $c'\in C^q(\Fg,M')$, with the property that
\begin{equation*}
d_{\rm CE}(c \cup c') = d_{\rm CE}c \cup c' + (-1)^p c\cup d_{\rm CE}c'.
\end{equation*}
Alternatively, if $c=m\ot \eta \in C^p(\Fg,M)$ and $c'=m'\ot \zeta \in C^q(\Fg,M')$, the cup product is given by
\begin{equation}
c\cup c' = m\cup m' \ot \eta \wg \z \in C^{p+q}(\Fg,P),
\end{equation}
see for instance \cite{CalaqueRossi-book}.

\ni In particular, the spaces $\Om_n^0$ and $\Om_n^1$ of formal differential forms are paired into $\Om_n^{\leq1}$, and thus the cohomology $H^\ast(W_n,\Om_n^{\leq1})$ possess a multiplicative structure, and a basis of $H^\ast(W_n,\Om_n^{\leq1})$ is given by $\lambda_k \in H^{2k-1}(W_n,\Om_n^0)$, $1 \leq k \leq n$, and $\mu \in H^1(W_n,\Om_n^1)$, subject to the relations
\begin{equation}\label{Vey-basis-cond}
\lambda_i\cup\lambda_j = -\lambda_j\cup \lambda_i, \qquad \lambda_k\cup\mu = \mu\cup\lambda_k.
\end{equation}
In particular, for $n=1$, the generators may be represented by
\begin{equation}\label{lambda}
\lambda(\xi) = \div(\xi),
\end{equation}
and
\begin{equation}\label{mu}
\mu(\xi) = d\div(\xi),
\end{equation}
see \cite[Thm. 2.2.7]{Fuks-book}.

\subsection{Lie algebra cohomology $H^\ast(\Fs\bowtie \Fn,\Om_n^{\leq1})$}

In this subsection we recall the bicomplex associated to the matched pair decomposition $W_n=\Fs\bowtie \Fn$, computing the Lie algebra cohomology of the Lie algebra $W_n$.

\ni Along the lines of \cite[Sect. 4.3]{RangSutl}, we consider the bicomplex
\begin{equation}\label{g-1-g-2-bicomplex}
\xymatrix{\vdots&\vdots&\vdots&\\
\Om_n^{\leq1}\ot \wdg^2\Fs^\ast \ar[r]^{\hD_{\rm CE}\;\;\;\;\,\,\,\,\,\,\,}\ar[u]^{\vDD_{\rm CE}}&\Om_n^{\leq1}\ot \wdg^2\Fs^\ast\ot \Fn^\ast \ar[r]^{\hD_{\rm CE}\;\;\;\;} \ar[u]^{\vDD_{\rm CE}}& \Om_n^{\leq1}\ot \wdg^2\Fs^\ast\ot \wdg^2\Fn^\ast \ar[r]^{\;\;\;\;\;\;\;\;\;\;\;\;\;\;\;\;\;\;\hD_{\rm CE}}\ar[u]^{\vDD_{\rm CE}}&\cdots\\
\Om_n^{\leq1}\ot \Fs^\ast \ar[r]^{\hD_{\rm CE}\,\,\,\,\,\,\,\,\,}\ar[u]^{\vDD_{\rm CE}}&\Om_n^{\leq1}\ot \Fs^\ast \ot \Fn^\ast \ar[r]^{\hD_{\rm CE}\,\,\,\,\,\,}\ar[u]^{\vDD_{\rm CE}} & \Om_n^{\leq1}\ot \Fs^\ast \ot \wdg^2\Fn^\ast \ar[r]^{\;\;\;\;\;\;\;\;\;\;\;\;\;\;\hD_{\rm CE}}\ar[u]^{\vDD_{\rm CE}}&\cdots\\
\Om_n^{\leq1} \ar[r]^{\hD_{\rm CE}}\ar[u]^{\vDD_{\rm CE}}& \Om_n^{\leq1}\ot \Fn^\ast \ar[r]^{\hD_{\rm CE}} \ar[u]^{\vDD_{\rm CE}}& \Om_n^{\leq1}\ot \wdg^2 \Fn^\ast\ar[r]^{\;\;\;\;\;\;\;\;\;\hD_{\rm CE}}\ar[u]^{\vDD_{\rm CE}}&\cdots}
\end{equation}

\ni The cohomology $H^\ast(W_n,\Om_n^{\leq1})$ can be computed by the total complex of the bicomplex \eqref{g-1-g-2-bicomplex}. This is achieved explicitly by
\begin{align}\label{natural-map}
& \natural:C^n(\Fs\bowtie \Fn, \Om_n^{\leq1})\lra {\rm Tot}^n(\Om_n^{\leq1}, \Fs^\ast,\Fn^\ast) \\\notag
& \natural(\Phi)(X_1,\ldots, X_p\mid \xi_1, \ldots, \xi_q)=\Phi(X_1\oplus 0,\ldots, X_p\oplus 0, 0\oplus\xi_1, \ldots, 0\oplus \xi_q),
\end{align}
whose inverse is given by
\begin{align*}
&\natural^{-1}(\om\ot \mu\ot\nu)(X_1\oplus\xi_1, \dots,X_{p+q}\oplus\xi_{p+q}) \\
&=\sum_{\s\in Sh(p,q)}(-1)^{\s}\om\mu(X_{\s(1)}, \dots,X_{\s(p)})\nu(\xi_{\s(p+1)}, \dots, \xi_{\s(p+q)}),
\end{align*}
where $Sh(p,q)$ denotes the set of $(p,q)$-shuffles. It follows from \cite[Lemma 2.7]{MoscRang11} that \eqref{natural-map} is an isomorphism of complexes.

\ni Finally, let us use \eqref{natural-map}, and its inverse, to carry the cup product construction \eqref{Lie-multp} on $C^\ast(\Fs\bowtie \Fn, \Om_n^{\leq1})$ to ${\rm Tot}^\ast(\Om_n^{\leq1}, \Fs^\ast,\Fn^\ast)$. Given any $a\ot \mu \ot \nu \in C^{p,q}(\Om_n^{\leq1}, \Fs^\ast,\Fn^\ast)$ in ${\rm Tot}^{p+q}(\Om_n^{\leq1}, \Fs^\ast,\Fn^\ast)$, and any $\om \ot \lambda \ot \rho \in C^{p',q'}(\Om_n^{\leq1}, \Fs^\ast,\Fn^\ast)$ in ${\rm Tot}^{p'+q'}(\Om_n^{\leq1}, \Fs^\ast,\Fn^\ast)$, we set
\begin{equation*}
(a\ot \mu \ot \nu) \cup (\om \ot \lambda \ot \rho) := \natural(\natural^{-1}(a\ot \mu \ot \nu) \cup \natural^{-1}(\om \ot \lambda \ot \rho)).
\end{equation*}
Accordingly,
\begin{align}\label{cup-total}
\begin{split}
& (a\ot \mu \ot \nu) \cup (\om \ot \lambda \ot \rho) := \natural(\natural^{-1}(a\ot \mu \ot \nu) \cup \natural^{-1}(\om \ot \lambda \ot \rho)) = \\
& \natural((a\ot \mu \wg \nu) \cup (\om \ot \lambda \wg \rho)) = \natural(a\om \ot \mu \wg \nu \wg \lambda \wg \rho)) = \\
& (-1)^{qp'}\natural (a\om \ot \mu\wg \lambda \wg \nu\wg \rho) = (-1)^{qp'} a\om \ot \mu\wg \lambda \ot \nu\wg \rho.
\end{split}
\end{align}

\section{Hopf-cyclic cohomology $HP(\Hc_n\cop,\Om_{n\d}^{\leq1})$}\label{Hopf-cyclic}

In this section we shall prove one of the main results of the paper, namely; a multiplicative structure on the bicomplex computing the Hopf-cyclic cohomology of $\Hc_n$.

\subsection{Hopf-cyclic bicomplex}

\ni Let $V$  be  a right-left SAYD module over $H$. Then,
\begin{equation}\label{aux-stand-Hopf-cyclic}
C(H,V) := \bigoplus_{q\geq 0} C^q(H,V), \quad C^q(H,V):= V\ot H^{\ot q}
\end{equation}
is a cocyclic module via
\begin{align*}
&\Fd_0(v\ot h^1\odots h^q)=v\ot 1\ot h^1\odots h^q,\\\notag
&\Fd_i(v\ot h^1\odots h^q)= v\ot h^1\odots h^i\ps{1}\ot h^i\ps{2}\odots h^q,  \\
&\Fd_{q+1}(v\ot h^1\odots h^q)=v\ns{0}\ot h^1\odots h^q\ot v\ns{-1},\\
&\Fs_j (v\ot h^1\odots h^q)= v\ot h^1\odots \ve(h^{j+1})\odots
h^q,\\
&\Ft (v\ot h^1\odots h^q)=v\ns{0}h^1\ps{1}\ot
S(h^1\ps{2})\cdot(h^2\odots h^q\ot v\ns{-1}).
\end{align*}
Using these operators one defines the Hochschild coboundary
\begin{equation*}
b: C^{q}_H(C,V)\ra C^{q+1}_H(C,V), \qquad b:=\sum_{i=0}^{q+1}(-1)^i\Fd_i, \\
\end{equation*}
and  the Connes boundary operator
\begin{align*}
&B: C^{q+1}_H(C,V)\ra C^q_H(C,V),\qquad B:=\left(\sum_{i=0}^{q}(-1)^{qi}\Ft_q^{i}\right) \Fs_{q}\Ft_{q+1}(\Id - (-1)^{q+1}\Ft_{q+1}).
\end{align*}
The cyclic cohomology of \eqref{aux-stand-Hopf-cyclic} is called the Hopf-cyclic cohomology of the Hopf algebra $H$, with coefficients in $V$, and it is denoted by $HC^\ast(H,V)$. Its periodized version, on the other hand, is denoted by $HP^\ast(H,V)$.

\ni Using the bicrossed-product structure of $\Hc_n$, it is shown in \cite{MoscRang09} that the Hopf-cyclic cohomology $HC^\ast(\Hc_n\cop, \Om_{n\d}^{\leq1})$ can be calculated by the diagonal subcomplex of the direct sum total of the bicomplex
\begin{align}\label{UF}
\begin{xy} \xymatrix{  \vdots\ar@<.6 ex>[d]^{\uparrow B} & \vdots\ar@<.6 ex>[d]^{\uparrow B}
 &\vdots \ar@<.6 ex>[d]^{\uparrow B} & &\\
\Om_{n\d}^{\leq1} \ot \Uc^{\ot 2} \ar@<.6 ex>[r]^{\hb}\ar@<.6
ex>[u]^{  \uparrow b  } \ar@<.6 ex>[d]^{\uparrow B}&
  \Om_{n\d}^{\leq1}\ot \Uc^{\ot 2}\ot \Fc   \ar@<.6 ex>[r]^{\hb}\ar@<.6 ex>[l]^{\hB}\ar@<.6 ex>[u]^{  \uparrow b  }
   \ar@<.6 ex>[d]^{\uparrow B}&\Om_{n\d}^{\leq1} \ot \Uc^{\ot 2}\ot\Fc^{\ot 2}
   \ar@<.6 ex>[r]^{~~\hb}\ar@<.6 ex>[l]^{\hB}\ar@<.6 ex>[u]^{  \uparrow b  }
   \ar@<.6 ex>[d]^{\uparrow B}&\ar@<.6 ex>[l]^{~~\hB} \hdots&\\
\Om_{n\d}^{\leq1} \ot \Uc \ar@<.6 ex>[r]^{\hb}\ar@<.6 ex>[u]^{  \uparrow b  }
 \ar@<.6 ex>[d]^{\uparrow B}&  \Om_{n\d}^{\leq1} \ot \Uc \ot\Fc \ar@<.6 ex>[r]^{\hb}
 \ar@<.6 ex>[l]^{\hB}\ar@<.6 ex>[u]^{  \uparrow b  } \ar@<.6 ex>[d]^{\uparrow B}
 &\Om_{n\d}^{\leq1}\ot \Uc \ot \Fc^{\ot 2}  \ar@<.6 ex>[r]^{~~\hb}\ar@<.6 ex>[l]^{\hB}\ar@<.6 ex>[u]^{  \uparrow b  }
  \ar@<.6 ex>[d]^{\uparrow B}&\ar@<.6 ex>[l]^{~~\hB} \hdots&\\
\Om_{n\d}^{\leq1}  \ar@<.6 ex>[r]^{\hb}\ar@<.6 ex>[u]^{  \uparrow b  }&
\Om_{n\d}^{\leq1}\ot\Fc \ar@<.6 ex>[r]^{\hb}\ar[l]^{\hB}\ar@<.6
ex>[u]^{  \uparrow b  }&\Om_{n\d}^{\leq1}\ot\Fc^{\ot 2}  \ar@<.6
ex>[r]^{~~\hb}\ar@<.6 ex>[l]^{\hB}\ar@<1 ex >[u]^{  \uparrow b  }
&\ar@<.6 ex>[l]^{~~\hB} \hdots& }
\end{xy}
\end{align}
with $\Uc:=U(\Fs)$, and $\Fc:=\Fc(N)$. The identification is given by
\begin{align}\label{PSI-1}\notag
& \Psi_{\acl}: D^n(U(\Fs),\Fc(N),\Om_{n\d}^{\leq1}) \lra C^n(\Hc_n\cop, \Om_{n\d}^{\leq1}), \\
&\Psi_{\acl}(\om\ot u^1 \ot \ldots u^n \ot f^1 \ot\cdots\ot f^n) =\\\notag
& \om\ot f^1\acl u^1\ns{0}\ot f^2u^1\ns{1}\acl u^2\ns{0} \odots f^n u^1\ns{n-1} \ldots u^{n-1}\ns{1}\acl u^n ,
\end{align}
whose inverse is
\begin{align}\label{PSI}\notag
&\Psi^{-1}_{\acl}:  C^n(\Hc_n\cop, \Om_{n\d}^{\leq1}) \lra D^n(U(\Fs),\Fc(N),\Om_{n\d}^{\leq1}), \\
&\Psi^{-1}_{\acl}(\om\ot   f^1\acl u^1\ot \ldots\ot f^n\acl u^n) = \\\notag
& \om\ot u^1\ns{0} \ot\cdots\ot u^{n-1}\ns{0}\ot u^n\ot f^1\ot\\\notag
& \hspace{1cm} f^2S(u^1\ns{n-1})\ot f^3S(u^1\ns{n-2}u^2\ns{n-2}) \ot\cdots\ot f^nS(u^1\ns{1} \dots u^{n-1}\ns{1}).
\end{align}
It follows from \cite[Prop. 4.4]{RangSutl} that the application of 
\begin{align} \label{antsym1}
\begin{split}
&\a: \Om_{n\d}^{\leq1} \ot \wg^{\ot\,p} \Fs\ot {\Fc(N)}^{\ot\,q} \lra \Om_{n\d}^{\leq1} \ot U(\Fs)^{\ot\,p}\ot {\Fc(N)}^{\ot\,q} \\
&\a(\om \ot X^1 \wg \ldots \wg X^p \ot f^1\ot\cdots\ot f^q)=  \frac{1}{p!} \sum_{\s\in S_p}(-1)^\s  \om \ot X^{\s(1)} \ot\cdots\ot X^{\s(p)} \ot f^1\ot\cdots\ot f^q
\end{split}
\end{align}
reduces the bicomplex \eqref{UF} to
\begin{align}\label{UF+}
\begin{xy} \xymatrix{  \vdots\ar[d]^{\p_{\rm CE}} & \vdots\ar[d]^{\p_{\rm CE}}
 &\vdots \ar[d]^{\p_{\rm CE}} & &\\
\Om_{n\d}^{\leq1} \ot \wg^2\Fs \ar[r]^{b_N\,\,\,\,\,\,\,\,\,\,\,\,\,\,\,\,\,\,} \ar[d]^{\p_{\rm CE}}& \Om_{n\d}^{\leq1}\ot \wg^2\Fs\ot{\Fc(N)}  \ar[r]^{b_N\,\,\,\,\,\,\,\,\,} \ar[d]^{\p_{\rm CE}}
 &\Om_{n\d}^{\leq1}\ot \wg^2\Fs\ot{\Fc(N)}^{\ot 2} \ar[r]^{\;\;\;\;\;\;\;\;\;\;\;\;\;\;\,\,\,\,\,\,\,\,\,\,\,\,\,\,\,\,\,\,b_N}   \ar[d]^{\p_{\rm CE}}& \hdots&\\
\Om_{n\d}^{\leq1} \ot \Fs \ar[r]^{b_N\,\,\,\,\,\,\,\,\,\,\,\,\,\,\,\,\,\,} \ar[d]^{\p_{\rm CE}}&  \Om_{n\d}^{\leq1}\ot \Fs\ot{\Fc(N)} \ar[r]^{b_N} \ar[d]^{\p_{\rm CE}}&\Om_{n\d}^{\leq1}\ot \Fs\ot{\Fc(N)}^{\ot 2}
   \ar[d]^{\p_{\rm CE}} \ar[r]^{\;\;\;\;\;\;\;\;\;\,\,\,\,\,\,\,\,\,\,\,\,\,\,\,\,\,\,\,\,\,\,\,\,\,\,b_N}& \hdots&\\
\Om_{n\d}^{\leq1} \ar[r]^{b_N}&  \Om_{n\d}^{\leq1}\ot{\Fc(N)} \ar[r]^{b_N}&\Om_{n\d}^{\leq1}\ot{\Fc(N)}^{\ot 2} \ar[r]^{\;\;\;\;\;\;b_N} & \hdots& }
\end{xy}
\end{align}
where
\begin{equation*}
\p_{\rm CE}: \Om_{n\d}^{\leq1}\ot \wg^p\Fs\ot{\Fc(N)}^{\ot q} \lra \Om_{n\d}^{\leq1}\ot \wg^{p-1}\Fs\ot{\Fc(N)}^{\ot q}
\end{equation*}
is the Lie algebra homology boundary of the Lie algebra $\Fs$, with coefficients in $\Om_{n\d}^{\leq1}\ot {\Fc(N)}^{\ot q}$, and
\begin{align*}
& b_N: \Om_{n\d}^{\leq1}\ot \wg^p\Fs\ot{\Fc(N)}^{\ot q} \lra \Om_{n\d}^{\leq1}\ot \wg^p \Fs\ot{\Fc(N)}^{\ot q+1} \\
& b_N(\om\ot \eta\ot f^1\ot\cdots\ot f^q) = \om\ot \eta\ot 1 \ot f^1\ot\cdots\ot f^q + \\
& \sum_{i=1}^q (-1)^i \om\ot \eta\ot f^1\ot\cdots\ot \D(f^i) \ot\cdots\ot f^q + \\
& (-1)^{q+1} \om\ns{0}\ot \eta\ns{0} \ot f^1\ot\cdots\ot f^q\ot S(\eta\ns{1})S(\om\ns{1}),
\end{align*}
see \cite[Prop. 4.4]{RangSutl}, or \cite[Prop. 3.21]{MoscRang09}. We recall here that the right $\Fc(N)$-coaction on $\Fs$ is given by \eqref{coact-g-F}, and it is extended to $\wg^\ast \Fs$ by multiplication.

\ni Finally, by the Poincar\'e duality, 
\begin{align}\label{Poincare-duality}
\begin{split}
& \FD_{\Fs}:\Om_{n\d}^{\leq1} \ot \wg^p \Fs^\ast \ot {\Fc(N)}^{\ot q} \lra \Om_{n\d}^{\leq1} \ot \wg^{n^2+n-p} \Fs \ot {\Fc(N)}^{\ot q} \\
& \FD_{\Fs}(\om\ot \eta \ot f^1\ot\cdots\ot f^q) = \om\ot \iota_\eta(\varpi) \ot f^1\ot\cdots\ot f^q,
\end{split}
\end{align}
where $\varpi\in \wg^{n^2+n}\Fs$ is the covolume form and $\iota_\eta:\wg^\ast \Fs\to \wg^{\ast-p}\Fs$ is the contruction by $\eta\in \wg^p\Fs^\ast$, we identify the total complex of the bicomplex \eqref{UF+} with that of 
\begin{align}\label{UF+*}
\begin{xy} \xymatrix{  \vdots & \vdots
 &\vdots &&\\
 \Om_{n\d}^{\leq1}\ot \wdg^2\Fs^\ast  \ar[u]^{d_{\rm CE}}\ar[r]^{b^\ast_N\,\,\,\,\,\,\,\,\,\,\;\;\;\;\;}&  \Om_{n\d}^{\leq1}\ot \wdg^2\Fs^\ast\ot{\Fc(N)}\ar[u]^{d_{\rm CE}} \ar[r]^{b^\ast_N\,\,\,\,\,\,\,\,\,\,}& \Om_{n\d}^{\leq1}\ot \wdg^2\Fs^\ast\ot{\Fc(N)}^{\ot 2} \ar[u]^{d_{\rm CE}} \ar[r]^{\;\;\;\;\;\;\;\;\;\;\;\;\;\,\,\,\,\,\,\,\,\,\,\,\,\,\,\,b^\ast_N} & \hdots&  \\
 \Om_{n\d}^{\leq1}\ot \Fs^\ast  \ar[u]^{d_{\rm CE}}\ar[r]^{b^\ast_N~~~~~}& \Om_{n\d}^{\leq1}\ot \Fs^\ast\ot{\Fc(N)} \ar[u]^{d_{\rm CE}} \ar[r]^{b^\ast_N}& \Om_{n\d}^{\leq1}\ot  \Fs^\ast\ot {\Fc(N)}^{\ot 2} \ar[u]^{d_{\rm CE}} \ar[r]^{\;\;\;\;\;\;\;\;\;\;\,\,\,\,\,\,\,\,\,\,\,\,\,\,\,b^\ast_N }& \hdots&  \\
   \Om_{n\d}^{\leq1}\ar[u]^{d_{\rm CE}}\ar[r]^{b^\ast_N~~~~~~~}& \Om_{n\d}^{\leq1}\ot {\Fc(N)} \ar[u]^{d_{\rm CE}}\ar[r]^{b^\ast_N}& \Om_{n\d}^{\leq1}\ot {\Fc(N)}^{\ot 2} \ar[u]^{d_{\rm CE}} \ar[r]^{\;\;\;\;\;\;\;\;\;\;b^\ast_N} & \hdots& }
\end{xy}
\end{align}
where 
\begin{align}\label{b_N-ast}
\begin{split}
& b^\ast_N: \Om_{n\d}^{\leq1}\ot \wg^p\Fs^\ast\ot{\Fc(N)}^{\ot q} \lra \Om_{n\d}^{\leq1}\ot \wg^p \Fs^\ast\ot{\Fc(N)}^{\ot q+1} \\
& b^\ast_N(\om\ot \eta\ot f^1\ot\cdots\ot f^q) = \om\ot \eta\ot 1 \ot f^1\ot\cdots\ot f^q + \\
& \sum_{i=1}^q (-1)^i \om\ot \eta\ot f^1\ot\cdots\ot \D(f^i) \odots f^q + \\
& (-1)^{q+1} \om\ns{0}\ot \eta\ns{0} \ot f^1\ot\cdots\ot f^q\ot S(\om\ns{1})\eta\ns{-1},
\end{split}
\end{align}
see \cite[Prop. 4.6]{RangSutl}, and
\begin{align}\label{d_CE}
\begin{split}
& d_{\rm CE}: \Om_{n\d}^{\leq1}\ot \wg^p\Fs^\ast\ot{\Fc(N)}^{\ot q} \lra \Om_{n\d}^{\leq1}\ot \wg^{p+1}\Fs^\ast\ot{\Fc(N)}^{\ot q} \\
& d_{\rm CE}(\om\ot \eta \ot \widetilde{f}) = \om \ot d_{\rm DR}(\eta) \ot \widetilde{f} - X_i\cdot\om\ot \upsilon^i\wg \eta \ot \widetilde{f} - \om\ot \upsilon^i\wg\eta \ot X_i\bullet\widetilde{f}
\end{split}
\end{align}
is the Lie algebra cohomology coboundary of the Lie algebra $\Fs$, with coefficients in $\Om_n^{\leq1}\ot {\Fc(N)}^{\ot q}$, see for instance \cite[(4.1)]{RangSutl}.
The map $d_{\rm DR}:\wg^p\Fs^\ast \to \wg^{p+1}\Fs^\ast$ is the deRham differential of forms, and the left $\Fc(N)$-coaction on $\wg^\ast \Fs^\ast$ is obtained by transposing the right $\Fc(N)$-coaction 
\begin{align}\label{F-coact-s}
\begin{split}
&\Fs^\ast \lra  \Fc(N)\ot \Fs^\ast, \\
&\t^i\mapsto 1 \ot \t^i, \qquad \t^i_j \mapsto 1 \ot \t^i_j + \eta^i_{jk}\ot \t^k,
\end{split}
\end{align}
which can be extended to $\wg^\ast \Fs^\ast$ multiplicatively.

\subsection{A multiplicative structure on $C^{\ast,\ast}(\Om_{n\d}^{\leq1},\Fs^\ast,\Fc(N))$}

In this subsection we introduce a multiplicative structure on the bicomplex
\begin{align}\label{bicomplex-Om-s-F}
\begin{split}
& C^{\ast,\ast}(\Om_{n\d}^{\leq1},\Fs^\ast,\Fc(N)) = \bigoplus_{p,q\geq 0} C^{p,q}(\Om_{n\d}^{\leq1},\Fs^\ast,\Fc(N)), \\
&\hspace{3cm} C^{p,q}(\Om_{n\d}^{\leq1},\Fs^\ast,\Fc(N)):=\Om_{n\d}^{\leq1} \ot\wg^p \Fs^\ast \ot \Fc(N)^{\ot\,q}
\end{split}
\end{align}
given by \eqref{UF+*}. To this end, for any $a\ot \eta \ot \widetilde{f} \in C^{p,q}(\Om_{n\d}^0,\Fs^\ast,\Fc(N))$, and $\om\ot \zeta \ot \widetilde{g} \in C^{p',q'}(\Om_{n\d}^1,\Fs^\ast,\Fc(N))$ let
\begin{equation}\label{cup-prod}
(a\ot \eta \ot \widetilde{f}) \cup (\om\ot \zeta \ot \widetilde{g}) := a\ns{0}\om\ot \eta\ns{0}\wg \zeta \ot \widetilde{f}\ot S(a\ns{1})\eta\ns{-1}\cdot \widetilde{g}.
\end{equation}

\begin{proposition}\label{prop-bN}
The horizontal coboundary \eqref{b_N-ast} acts as a graded derivation, \ie for any $a\ot \eta \ot \widetilde{f} \in C^{p,q}(\Om_{n\d}^0,\Fs^\ast,\Fc(N))$, and $\om\ot \zeta \ot \widetilde{g} \in C^{p',q'}(\Om_{n\d}^1,\Fs^\ast,\Fc(N))$,
\begin{align*}
& b^\ast_N\Big((a\ot \eta \ot \widetilde{f}) \cup (\om\ot \zeta \ot \widetilde{g})\Big) = \\
&\hspace{1cm} b^\ast_N\Big(a\ot \eta \ot \widetilde{f}\Big) \cup (\om\ot \zeta \ot \widetilde{g}) + (-1)^q (a\ot \eta \ot \widetilde{f}) \cup b^\ast_N\Big(\om\ot \zeta \ot \widetilde{g}\Big).
\end{align*}
\end{proposition}

\begin{proof}
Given $a\ot \eta \ot \widetilde{f} \in C^{p,q}(\Om_{n\d}^0,\Fs^\ast,\Fc(N))$, let 
\begin{equation*}
\D_i(\widetilde{f}) := f^1\odots \D(f^i) \odots f^q.
\end{equation*}
We then note that
\begin{align*}
& b^\ast_N\Big((a\ot \eta \ot \widetilde{f}) \cup (\om\ot \zeta \ot \widetilde{g})\Big) =\\
& a\ns{0}\om\ot \eta\ns{0}\wg \zeta \ot 1 \ot \widetilde{f}\ot S(a\ns{1})\eta\ns{-1}\cdot \widetilde{g} + \\
& \sum_{i=1}^q(-1)^i a\ns{0}\om\ot \eta\ns{0}\wg \zeta \ot \D_i(\widetilde{f}) \ot S(a\ns{1})\eta\ns{-1}\cdot \widetilde{g} + \\
& \sum_{i=q+1}^{q+q'}(-1)^i a\ns{0}\om\ot \eta\ns{0}\wg \zeta \ot \widetilde{f} \ot S(a\ns{1})\eta\ns{-1}\cdot \D_i(\widetilde{g}) + \\
& (-1)^{q+q'+1} (a\ns{0}\om)\ns{0}\ot \eta\ns{0}\wg \zeta\ns{0} \ot \widetilde{f}\ot \\
&\hspace{3cm} S(a\ns{1})\eta\ns{-2}\cdot \widetilde{g} \ot S((a\ns{0}\om)\ns{1})\eta\ns{-1}\zeta\ns{-1},
\end{align*}
which can be rewritten as
\begin{align*}
& a\ns{0}\om\ot \eta\ns{0}\wg \zeta \ot 1 \ot \widetilde{f}\ot S(a\ns{1})\eta\ns{-1}\cdot \widetilde{g} + \\
& \sum_{i=1}^q(-1)^i a\ns{0}\om\ot \eta\ns{0}\wg \zeta \ot \D_i(\widetilde{f}) \ot S(a\ns{1})\eta\ns{-1}\cdot \widetilde{g} + \\
& (-1)^{q+1} a\ns{0}\ns{0}\om \ot \eta\ns{0}\ns{0} \wg \zeta \ot \widetilde{f} \ot \\
&\hspace{3cm} S(a\ns{1})\eta\ns{-1} \ot S(a\ns{0}\ns{1})\eta\ns{0}\ns{-1}\cdot\widetilde{g}) + \\
& (-1)^q \Big\{a\ns{0}\om \ot \eta\ns{0} \wg \zeta \ot \widetilde{f} \ot S(a\ns{1})\eta\ns{-1}\cdot (1 \ot \widetilde{g}) + \\ 
& \sum_{i=1}^{q'}(-1)^i a\ns{0}\om\ot \eta\ns{0}\wg \zeta \ot \widetilde{f} \ot S(a\ns{1})\eta\ns{-1}\cdot \D_i(\widetilde{g}) + \\
& (-1)^{q'+1} (a\om)\ns{0}\ot \eta\ns{0}\wg \zeta\ns{0} \ot \widetilde{f}\ot \\
&\hspace{3cm} S(a\ns{1})\eta\ns{-2}\cdot \widetilde{g} \ot S((a\ns{0}\om)\ns{1})\eta\ns{-1}\zeta\ns{-1}\Big\},
\end{align*}
the first three lines of which being $b^\ast_N\Big(a\ot \eta \ot \widetilde{f}\Big) \cup (\om\ot \zeta \ot \widetilde{g})$, and the last three being $(-1)^q (a\ot \eta \ot \widetilde{f}) \cup b^\ast_N\Big(\om\ot \zeta \ot \widetilde{g}\Big)$.
\end{proof}

\noindent On the next move we deal with the vertical coboundary. To this end, we record a series of lemmas below. The first one is about the action \eqref{bullet}.

\begin{lemma}\label{lemma-bullet}
For any $X\in \Fs$, any $\widetilde{f} \in \Fc(N)^{\ot\,r}$, and any $\widetilde{g} \in \Fc(N)^{\ot\,s}$, 
\begin{equation*}
X\bullet (\widetilde{f}\ot \widetilde{g}) = X\ns{0}\bullet \widetilde{f}\ot X\ns{1}\cdot \widetilde{g} + \widetilde{f}\ot X\bullet \widetilde{g}.
\end{equation*}
\end{lemma}

\begin{proof}
It follows at once from the definition of the action \eqref{bullet} that
\begin{align*}
& X\bullet (f^1\odots f^q) = (1\acl X)\cdot (f^1\odots f^q) = \\
& X\ns{0} \rt f^1 \ot X\ns{1}\cdot (f^2\odots f^q) + f^1\ot X\bullet (f^2\odots f^q).
\end{align*}
The claim then follows immediately.
\end{proof}

\noindent The rest of the lemmas point out some auxiliary results by the commutativity of the horizontal and the vertical coboundary maps of the bicomplex \eqref{UF+*}.

\begin{lemma}\label{lemma-DR}
For any $\eta \in \wg^p\Fs^\ast$,
\begin{equation*}
d_{\rm DR}(\eta)\ns{-1} \ot d_{\rm DR}(\eta)\ns{0} = \eta\ns{-1}\ot d_{\rm DR}(\eta\ns{0}) - X_i\rt \eta\ns{-1}\ot \upsilon^i\wg \eta\ns{0}.
\end{equation*}
\end{lemma}

\begin{proof}
We use the commutativity of the horizontal and the vertical coboundary maps of \eqref{UF+*}. Fixing a trivial coefficient, we have on one hand
\begin{equation*}
b_N^\ast d_{\rm CE}(\eta) = b_N^\ast(d_{\rm DR}(\eta)) = d_{\rm DR}(\eta) \ot 1 - d_{\rm DR}(\eta)\ns{0} \ot d_{\rm DR}(\eta)\ns{-1},
\end{equation*}
and on the other hand
\begin{align*}
& d_{\rm CE}b_N^\ast(\eta) = d_{\rm CE}(\eta \ot 1 - \eta\ns{0}\ot \eta\ns{-1}) = \\
&\hspace{2cm} d_{\rm DR}(\eta) \ot 1 - d_{\rm DR}(\eta\ns{0}) \ot \eta\ns{-1} + \upsilon^i\wg \eta\ns{0} \ot X_i\rt\eta\ns{-1}.
\end{align*}
The result follows.
\end{proof}

\begin{lemma}\label{lemma-X-dot-a}
For any $a\in \Om_n^0$, and any $\eta\in \wg^p\Fs^\ast$,
\begin{align*}
& (X_i\cdot a)\ns{0} \ot (\upsilon^i\wg \eta)\ns{0} \ot S((X_i\cdot a)\ns{1})(\upsilon^i\wg \eta)\ns{-1} = \\
& X_i\cdot a\ns{0} \ot \upsilon^i\wg \eta\ns{0} \ot S(a\ns{1})\eta\ns{-1} +  a\ns{0} \ot \upsilon^i\wg \eta\ns{0} \ot X_i\rt S(a\ns{1})\eta\ns{-1}.
\end{align*} 
\end{lemma}

\begin{proof}
On one hand we have
\begin{align*}
& b_N^\ast d_{\rm CE} (a\ot \eta) = b_N^\ast\Big(a\ot d_{\rm DR}(\eta) - X_i\cdot a \ot \upsilon^i\wg \eta\Big) = \\
& a\ot d_{\rm DR}(\eta) \ot 1 - a\ns{0}\ot d_{\rm DR}(\eta\ns{0}) \ot S(a\ns{1})\eta\ns{-1} + \\
& a\ns{0}\ot \upsilon^i\wg \eta\ns{0} \ot S(a\ns{1})(X_i\rt \eta\ns{-1}) - X_i\cdot a \ot \upsilon^i\wg \eta + \\
& (X_i\cdot a)\ns{0} \ot (\upsilon^i\wg \eta)\ns{0} \ot S((X_i\cdot a)\ns{1})(\upsilon^i\wg \eta)\ns{-1},
\end{align*}
where we used Lemma \ref{lemma-DR} on the second equality, and on the other hand,
\begin{align*}
& d_{\rm CE} b_N^\ast (a\ot \eta) = d_{\rm CE} \Big(a\ot \eta\ot 1 - a\ns{0}\ot \eta\ns{0} \ot S(a\ns{1}) \eta\ns{-1}\Big) = \\
& a\ot d_{\rm DR}(\eta)\ot 1 - X_i\cdot a\ot \upsilon^i\wg \eta\ot 1 - a\ns{0}\ot d_{\rm DR}(\eta\ns{0}) \ot S(a\ns{1}) \eta\ns{-1} + \\
& X_i\cdot a\ns{0}\ot \upsilon^i\wg \eta\ns{0} \ot S(a\ns{1}) \eta\ns{-1} + a\ns{0}\ot \upsilon^i\wg \eta\ns{0} \ot X_i\rt (S(a\ns{1}) \eta\ns{-1}).
\end{align*}
In view of the commutativity of the horizontal and the vertical coboundaries of the bicomplex \eqref{UF+*}, a comparison of the two equations yields the claim.
\end{proof}

\begin{lemma}\label{lemma-upsilon}
For any $\widetilde{f}\in \Fc(N)^{\ot\,q}$,
\begin{equation*}
\upsilon^i\ns{0} \ot X_i\bullet \widetilde{f} \ot \upsilon^i\ns{-1} = \upsilon^i \ot X_i\bullet \widetilde{f} \ot 1.
\end{equation*}
\end{lemma}

\begin{proof}
We have
\begin{align*}
& b_N^\ast d_{\rm CE}(\widetilde{f}) = b_N^\ast\Big(-\upsilon^i\ot X_i\bullet \widetilde{f}\Big) = -\upsilon^i\ot 1 \ot X_i\bullet \widetilde{f} + \upsilon^i\ns{0} \ot X_i\bullet \widetilde{f} \ot \upsilon^i\ns{-1},
\end{align*}
and 
\begin{align*}
& d_{\rm CE}b_N^\ast(\widetilde{f}) = d_{\rm CE}\Big(1 \ot \widetilde{f} - \widetilde{f} \ot 1\Big) =  - \upsilon^i \ot X_i\bullet (1 \ot \widetilde{f}) + \upsilon^i\ot X_i\bullet (\widetilde{f}\ot 1) = \\
& - \upsilon^i \ot 1 \ot X_i\bullet  \widetilde{f} + \upsilon^i\ot X_i\bullet \widetilde{f}\ot 1,
\end{align*}
where we used Lemma \ref{lemma-bullet} on the last equation. The result then follows once again by the commutativity of the horizontal and the vertical coboundaries of \eqref{UF+*}.
\end{proof}

\begin{proposition}\label{prop-dCE}
The vertical coboundary \eqref{d_CE} acts as a graded differential, \ie for any $a\ot \eta \ot \widetilde{f} \in C^{p,q}(\Om_{n\d}^0,\Fs^\ast,\Fc(N))$, and $\om\ot \zeta \ot \widetilde{g} \in C^{p',q'}(\Om_{n\d}^1,\Fs^\ast,\Fc(N))$,
\begin{align*}
& d_{\rm CE}\Big((a\ot \eta \ot \widetilde{f}) \cup (\om\ot \zeta \ot \widetilde{g})\Big) = \\
&\hspace{1cm} d_{\rm CE}\Big(a\ot \eta \ot \widetilde{f}\Big) \cup (\om\ot \zeta \ot \widetilde{g}) + (-1)^p (a\ot \eta \ot \widetilde{f}) \cup d_{\rm CE}\Big(\om\ot \zeta \ot \widetilde{g}\Big).
\end{align*}
\end{proposition}

\begin{proof}
We first observe that
\begin{align*}
& d_{\rm CE}\Big((a\ot \eta \ot \widetilde{f}) \cup (\om\ot \zeta \ot \widetilde{g})\Big) = \\
& a\ns{0}\om\ot d_{\rm DR}\Big(\eta\ns{0}\wg \zeta\Big) \ot \widetilde{f}\ot S(a\ns{1})\eta\ns{-1}\cdot \widetilde{g} + \\
& \hspace{0cm}- X_i \cdot (a\ns{0}\om)\ot \upsilon^i \wg \eta\ns{0}\wg \zeta \ot \widetilde{f}\ot S(a\ns{1})\eta\ns{-1}\cdot \widetilde{g} + \\
& - a\ns{0}\om\ot \upsilon^i\wg\eta\ns{0}\wg \zeta \ot X_i\bullet(\widetilde{f}\ot S(a\ns{1})\eta\ns{-1}\cdot \widetilde{g}).
\end{align*}
Using the fact that the deRham coboundary is a graded differential we arrive at
\begin{align*}
& d_{\rm CE}\Big((a\ot \eta \ot \widetilde{f}) \cup (\om\ot \zeta \ot \widetilde{g})\Big) = \\
& a\ns{0}\om\ot d_{\rm DR}\Big(\eta\ns{0}\Big)\wg \zeta \ot \widetilde{f}\ot S(a\ns{1})\eta\ns{-1}\cdot \widetilde{g} + \\
& (-1)^p a\ns{0}\om\ot \eta\ns{0}\wg d_{\rm DR}\Big(\zeta\Big) \ot \widetilde{f}\ot S(a\ns{1})\eta\ns{-1}\cdot \widetilde{g} + \\
& \hspace{1.5cm}- X_i \cdot (a\ns{0}\om)\ot \upsilon^i \wg \eta\ns{0}\wg \zeta \ot \widetilde{f}\ot S(a\ns{1})\eta\ns{-1}\cdot \widetilde{g} + \\
& - a\ns{0}\om\ot \upsilon^i\wg\eta\ns{0}\wg \zeta \ot X_i\bullet(\widetilde{f}\ot S(a\ns{1})\eta\ns{-1}\cdot \widetilde{g}).
\end{align*}
Next, we recall that the Lie algebra $\Fs\subseteq W_n$ atcs on $\Om_n^0$ by derivations. Thus,
\begin{align*}
& d_{\rm CE}\Big((a\ot \eta \ot \widetilde{f}) \cup (\om\ot \zeta \ot \widetilde{g})\Big) = \\
& a\ns{0}\om\ot d_{\rm DR}\Big(\eta\ns{0}\Big)\wg \zeta \ot \widetilde{f}\ot S(a\ns{1})\eta\ns{-1}\cdot \widetilde{g} + \\
& (-1)^p a\ns{0}\om\ot \eta\ns{0}\wg d_{\rm DR}\Big(\zeta\Big) \ot \widetilde{f}\ot S(a\ns{1})\eta\ns{-1}\cdot \widetilde{g} + \\
& \hspace{1.5cm}- (X_i \cdot a\ns{0})\om\ot \upsilon^i \wg \eta\ns{0}\wg \zeta \ot \widetilde{f}\ot S(a\ns{1})\eta\ns{-1}\cdot \widetilde{g} + \\
& \hspace{1.5cm}- a\ns{0}(X_i \cdot \om)\ot \upsilon^i \wg \eta\ns{0}\wg \zeta \ot \widetilde{f}\ot S(a\ns{1})\eta\ns{-1}\cdot \widetilde{g} + \\
& - a\ns{0}\om\ot \upsilon^i\wg\eta\ns{0}\wg \zeta \ot X_i\bullet(\widetilde{f}\ot S(a\ns{1})\eta\ns{-1}\cdot \widetilde{g}).
\end{align*}
Then using Lemma \ref{lemma-bullet} we get
\begin{align*}
& d_{\rm CE}\Big((a\ot \eta \ot \widetilde{f}) \cup (\om\ot \zeta \ot \widetilde{g})\Big) = \\
& a\ns{0}\om\ot d_{\rm DR}\Big(\eta\ns{0}\Big)\wg \zeta \ot \widetilde{f}\ot S(a\ns{1})\eta\ns{-1}\cdot \widetilde{g} + \\
& (-1)^p a\ns{0}\om\ot \eta\ns{0}\wg d_{\rm DR}\Big(\zeta\Big) \ot \widetilde{f}\ot S(a\ns{1})\eta\ns{-1}\cdot \widetilde{g} + \\
& \hspace{1.5cm}- (X_i \cdot a\ns{0})\om\ot \upsilon^i \wg \eta\ns{0}\wg \zeta \ot \widetilde{f}\ot S(a\ns{1})\eta\ns{-1}\cdot \widetilde{g} + \\
& \hspace{1.5cm}- a\ns{0}(X_i \cdot \om)\ot \upsilon^i \wg \eta\ns{0}\wg \zeta \ot \widetilde{f}\ot S(a\ns{1})\eta\ns{-1}\cdot \widetilde{g} + \\
& - a\ns{0}\om\ot \upsilon^i\wg\eta\ns{0}\wg \zeta \ot {X_i}\ns{0}\bullet \widetilde{f}\ot X_i\ns{1}S(a\ns{1})\eta\ns{-1}\cdot \widetilde{g} + \\
& - a\ns{0}\om\ot \upsilon^i\wg\eta\ns{0}\wg \zeta \ot \widetilde{f}\ot X_i\bullet(S(a\ns{1})\eta\ns{-1}\cdot \widetilde{g}),
\end{align*}
where
\begin{align*}
& a\ns{0}\om\ot \upsilon^i\wg\eta\ns{0}\wg \zeta \ot \widetilde{f}\ot X_i\bullet(S(a\ns{1})\eta\ns{-1}\cdot \widetilde{g}) = \\
& a\ns{0}\om\ot \upsilon^i\wg\eta\ns{0}\wg \zeta \ot \widetilde{f}\ot (1 \acl X_i)(S(a\ns{1})\eta\ns{-1}\acl 1)\cdot \widetilde{g} = \\
& a\ns{0}\om\ot \upsilon^i\wg\eta\ns{0}\wg \zeta \ot \widetilde{f}\ot (X_i\rt S(a\ns{1})\eta\ns{-1})\cdot \widetilde{g} + \\
& a\ns{0}\om\ot \upsilon^i\wg\eta\ns{0}\wg \zeta \ot \widetilde{f}\ot S(a\ns{1})\eta\ns{-1}\cdot (X_i\bullet \widetilde{g}).
\end{align*} 
As a result,
\begin{align}\label{d(cup)}
\begin{split}
& d_{\rm CE}\Big((a\ot \eta \ot \widetilde{f}) \cup (\om\ot \zeta \ot \widetilde{g})\Big) = \\
& a\ns{0}\om\ot d_{\rm DR}\Big(\eta\ns{0}\Big)\wg \zeta \ot \widetilde{f}\ot S(a\ns{1})\eta\ns{-1}\cdot \widetilde{g} + \\
& (-1)^p a\ns{0}\om\ot \eta\ns{0}\wg d_{\rm DR}\Big(\zeta\Big) \ot \widetilde{f}\ot S(a\ns{1})\eta\ns{-1}\cdot \widetilde{g} + \\
& \hspace{1.5cm}- (X_i \cdot a\ns{0})\om\ot \upsilon^i \wg \eta\ns{0}\wg \zeta \ot \widetilde{f}\ot S(a\ns{1})\eta\ns{-1}\cdot \widetilde{g} + \\
& \hspace{1.5cm}- a\ns{0}(X_i \cdot \om)\ot \upsilon^i \wg \eta\ns{0}\wg \zeta \ot \widetilde{f}\ot S(a\ns{1})\eta\ns{-1}\cdot \widetilde{g} + \\
& - a\ns{0}\om\ot \upsilon^i\wg\eta\ns{0}\wg \zeta \ot {X_i}\ns{0}\bullet \widetilde{f}\ot X_i\ns{1}S(a\ns{1})\eta\ns{-1}\cdot \widetilde{g} + \\
& - a\ns{0}\om\ot \upsilon^i\wg\eta\ns{0}\wg \zeta \ot \widetilde{f}\ot (X_i\rt S(a\ns{1})\eta\ns{-1})\cdot \widetilde{g} + \\
& - a\ns{0}\om\ot \upsilon^i\wg\eta\ns{0}\wg \zeta \ot \widetilde{f}\ot S(a\ns{1})\eta\ns{-1}\cdot (X_i\bullet \widetilde{g}).
\end{split}
\end{align}
On the other hand,
\begin{align*}
& d_{\rm CE}\Big(a\ot \eta \ot \widetilde{f}\Big) \cup (\om\ot \zeta \ot \widetilde{g}) = \\
& (a\ot d_{\rm DR}(\eta) \ot \widetilde{f} - X_i\rt a\ot \upsilon^i\wg\eta \ot \widetilde{f} - a\ot \upsilon^i\wg\eta \ot X_i\bullet\widetilde{f}) \cup (\om\ot \zeta \ot \widetilde{g}) = \\
& a\ns{0}\om\ot d_{\rm DR}(\eta)\ns{0} \wg \zeta \ot \widetilde{f} \ot S(a\ns{1})d_{\rm DR}(\eta)\ns{-1} \cdot \widetilde{g} + \\
& \hspace{1cm}- (X_i\cdot a)\ns{0}\om \ot (\upsilon^i\wg\eta)\ns{0} \wg \zeta \ot \widetilde{f} \ot S((X_i\rt a)\ns{1})(\upsilon^i\wg\eta)\ns{-1}\cdot \widetilde{g} + \\
& - a\ns{0}\om\ot (\upsilon^i\wg\eta)\ns{0} \ot X_i\bullet\widetilde{f} \ot S(a\ns{1})(\upsilon^i\wg\eta)\ns{-1}\cdot \widetilde{g},
\end{align*}
from which we arrive, in view of Lemma \ref{lemma-DR}, at
\begin{align*}
& d_{\rm CE}\Big(a\ot \eta \ot \widetilde{f}\Big) \cup (\om\ot \zeta \ot \widetilde{g}) = \\
& a\ns{0}\om\ot d_{\rm DR}(\eta\ns{0}) \wg \zeta \ot \widetilde{f} \ot S(a\ns{1})\eta\ns{-1} \cdot \widetilde{g} + \\
& - a\ns{0}\om\ot \upsilon^i \wg \eta\ns{0} \wg \zeta \ot \widetilde{f} \ot S(a\ns{1})(X_i\rt \eta\ns{-1}) \cdot \widetilde{g} + \\
& \hspace{.5cm}- (X_i\cdot a)\ns{0}\om \ot (\upsilon^i\wg\eta)\ns{0} \wg \zeta \ot \widetilde{f} \ot S((X_i\rt a)\ns{1})(\upsilon^i\wg\eta)\ns{-1}\cdot \widetilde{g} + \\
& - a\ns{0}\om\ot (\upsilon^i\wg\eta)\ns{0} \ot X_i\bullet\widetilde{f} \ot S(a\ns{1})(\upsilon^i\wg\eta)\ns{-1}\cdot \widetilde{g}.
\end{align*}
Invoking next Lemma \ref{lemma-X-dot-a} and Lemma \ref{lemma-upsilon},
\begin{align}\label{d-cup}
\begin{split}
& d_{\rm CE}\Big(a\ot \eta \ot \widetilde{f}\Big) \cup (\om\ot \zeta \ot \widetilde{g}) = \\
& a\ns{0}\om\ot d_{\rm DR}(\eta\ns{0}) \wg \zeta \ot \widetilde{f} \ot S(a\ns{1})\eta\ns{-1} \cdot \widetilde{g} + \\
& - a\ns{0}\om\ot \upsilon^i \wg \eta\ns{0} \wg \zeta \ot \widetilde{f} \ot S(a\ns{1})(X_i\rt \eta\ns{-1}) \cdot \widetilde{g} + \\
& \hspace{1cm}- (X_i\cdot a\ns{0})\om \ot \upsilon^i\wg\eta\ns{0} \wg \zeta \ot \widetilde{f} \ot S(a\ns{1})\eta\ns{-1}\cdot \widetilde{g} + \\
& \hspace{1cm}- a\ns{0}\om \ot (\upsilon^i\wg\eta)\ns{0} \wg \zeta \ot \widetilde{f} \ot (X_i\rt S(a\ns{1}))\eta\ns{-1}\cdot \widetilde{g} + \\
& - a\ns{0}\om\ot \upsilon^i\wg\eta\ns{0} \ot X_i\bullet\widetilde{f} \ot S(a\ns{1})\eta\ns{-1}\cdot \widetilde{g}.
\end{split}
\end{align}
Finally we see that
\begin{align}\label{cup-d}
\begin{split}
& (a\ot \eta \ot \widetilde{f}) \cup d_{\rm CE}\Big(\om\ot \zeta \ot \widetilde{g}\Big) = \\
& (a\ot \eta \ot \widetilde{f}) \cup (\om\ot d_{\rm DR}(\zeta) \ot \widetilde{g} - X_i\cdot\om\ot \upsilon^i\wg\zeta \ot \widetilde{g} - \om\ot \upsilon^i\wg\zeta \ot X_i\bullet\widetilde{g}) = \\
& a\ns{0}\om\ot \eta\ns{0}\wg d_{\rm DR}(\zeta) \ot S(a\ns{1})\eta\ns{-1}\cdot\widetilde{g} + \\
& \hspace{1cm}- a\ns{0}(X_i\cdot\om)\ot \eta\ns{0}\wg\upsilon^i\wg\zeta \ot S(a\ns{1})\eta\ns{-1}\cdot\widetilde{g} + \\
& -a\ns{0}\om\ot \eta\ns{0}\wg\upsilon^i\wg\zeta \ot S(a\ns{1})\eta\ns{-1}\cdot(X_i\bullet\widetilde{g}).
\end{split}
\end{align}
The claim now follows immediately from the comparison of \eqref{d(cup)}, \eqref{d-cup} and \eqref{cup-d}.
\end{proof}

\ni We are ready to express the main result of the section.

\begin{theorem}
The coboundary
\begin{equation*}
d_{\rm CE} + (-1)^pb_N^\ast:C^{p,q}(\Om_{n\d}^{\leq1},\Fs^\ast,\Fc(N)) \to C^{p+1,q}(\Om_{n\d}^{\leq1},\Fs^\ast,\Fc(N)) \oplus C^{p,q+1}(\Om_{n\d}^{\leq1},\Fs^\ast,\Fc(N))
\end{equation*}
of the total complex of the bicomplex \eqref{UF+*} acts as a graded differential with respect to the product structure given by
\begin{equation*}
(a\ot \eta \ot \widetilde{f}) \ast (\om\ot \zeta \ot \widetilde{g}) = (-1)^{qp'}(a\ot \eta \ot \widetilde{f}) \cup (\om\ot \zeta \ot \widetilde{g})
\end{equation*}
for $a\ot \eta \ot \widetilde{f} \in \Om_{n\d}^0\ot \wg^p\Fs^\ast\ot \Fc(N)^{\ot\,q}$, and $\om\ot \zeta \ot \widetilde{g} \in \Om_{n\d}^1\ot \wg^{p'}\Fs^\ast\ot \Fc(N)^{\ot\,q'}$.
\end{theorem}

\begin{proof}
We have
\begin{align*}
& (d_{\rm CE} + (-1)^{p+p'}b_N^\ast)\Big((a\ot \eta \ot \widetilde{f}) \ast (\om\ot \zeta \ot \widetilde{g})\Big) = \\
& (-1)^{qp'}d_{\rm CE}\Big((a\ot \eta \ot \widetilde{f}) \cup (\om\ot \zeta \ot \widetilde{g})\Big) + \\
& \hspace{2cm} (-1)^{qp'+p+p'}b_N^\ast\Big((a\ot \eta \ot \widetilde{f}) \cup (\om\ot \zeta \ot \widetilde{g})\Big).
\end{align*}
In view of Proposition \ref{prop-bN}, and Proposition \ref{prop-dCE}
\begin{align*}
& (d_{\rm CE} + (-1)^{p+p'}b_N^\ast)\Big((a\ot \eta \ot \widetilde{f}) \ast (\om\ot \zeta \ot \widetilde{g})\Big) = \\
& (-1)^{qp'}d_{\rm CE}\Big(a\ot \eta \ot \widetilde{f}\Big) \cup (\om\ot \zeta \ot \widetilde{g}) + \\
& (-1)^{qp' + p + p'}(a\ot \eta \ot \widetilde{f}) \cup d_{\rm CE}\Big(\om\ot \zeta \ot \widetilde{g}\Big) + \\
& \hspace{2cm} (-1)^{qp'+p}b_N^\ast\Big(a\ot \eta \ot \widetilde{f}\Big) \cup (\om\ot \zeta \ot \widetilde{g}) + \\
& \hspace{2cm} (-1)^{qp'+p+p'+q}(a\ot \eta \ot \widetilde{f}) \cup b_N^\ast\Big(\om\ot \zeta \ot \widetilde{g}\Big) = \\
& (d_{\rm CE} + (-1)^pb_N^\ast)\Big(a\ot \eta \ot \widetilde{f}\Big) \ast (\om\ot \zeta \ot \widetilde{g}) + \\
& (-1)^{p+q}(a\ot \eta \ot \widetilde{f}) \ast (d_{\rm CE} + (-1)^{p'}b_N^\ast)\Big(\om\ot \zeta \ot \widetilde{g}\Big).
\end{align*}
\end{proof}

\section{The transfer of classes}\label{classes}

\subsection{Multiplicativity of the characteristic homomorphism}

We show that the chacracteristic homomorphism of \cite[Thm. 4.10]{RangSutl}, identifying the Hopf-cyclic cohomology of $\Hc_n$ with the Lie algebra cohomology of $W_n$, with nontrivial coefficients, respects the multiplicative structures on its domain and the range.

\begin{theorem}\label{thm-char-homm}
For the Lie algebra $W_n = \Fs \bowtie \Fn$, the Hopf algebra $\Fc(N) \acl U(\Fs)$, and the induced $\Fc(N)\acl U(\Fs)$-module $\Om_n^{\leq1}$, 
\begin{equation*}
HP^\ast(\Fc(N) \acl U(\Fs), \Om_{n\d}^{\leq1}) \cong \bigoplus_{m = \ast\,\,{\rm mod}\,2} H^m(W_n,\Om_n^{\leq1}).
\end{equation*}
\end{theorem}

\begin{proof}
In view of \cite[Thm. 4.10]{RangSutl}, we need to show that the van Est type map 
\begin{align}\label{VE}
\begin{split}
&\Vc: \Om_n^{\leq1}\ot \wdg^p\Fs^\ast\ot {\Fc(N)}^{\ot q}\lra \Om_n^{\leq1}\ot\wdg^p \Fs^\ast\ot \wdg^q \Fn^\ast \\
&\Vc(\om\ot\eta\ot f^1\odots f^q)( X_1,\ldots,X_p\mid \xi_1,\ldots, \xi_q) = \\
& \hspace{1cm} \eta(X_1,\ldots, X_p)\sum_{\s\in S_q}(-1)^\s \langle \xi_{\s(1)}\,,\, f^1\rangle \ldots \langle \xi_{\s(q)}\,,\, f^q\rangle \om
\end{split}
\end{align}
from the total complex of the bicomplex \eqref{UF+*} to that of \eqref{g-1-g-2-bicomplex} is a quasi-isomorphism. This, in turn, follows at once from \cite[Lemma 4.1]{RangSutl} given the non-degenerate pairing \cite[(3.50)]{RangSutl-III}, see also \cite{ConnMosc98}, between $\Fc(N)$ and $U(\Fn)$.
\end{proof}

\ni The following is our main result.

\begin{theorem}\label{VE-multp}
The quasi-isomorphism \eqref{VE} is multiplicative, \ie
\begin{equation*}
\Vc\Big((a\ot \eta \ot \widetilde{f}) \ast (\om\ot \zeta \ot \widetilde{g})\Big) = \Vc(a\ot \eta \ot \widetilde{f}) \cup \Vc(\om\ot \zeta \ot \widetilde{g}).
\end{equation*}
\end{theorem}

\begin{proof}
For an $a \ot \eta \ot \widetilde{f} \in C^{p,q}(\Om_{n\d}^{\leq1},\Fs^\ast,\Fc(N))$, and an $\om\ot \zeta \ot \widetilde{g} \in C^{p',q'}(\Om_{n\d}^{\leq1},\Fs^\ast,\Fc(N))$, we observe that
\begin{align*}
& \Vc\Big((a\ot \eta \ot \widetilde{f}) \ast (\om\ot \zeta \ot \widetilde{g})\Big)( X_1,\ldots,X_{p+p'}\mid \xi_1,\ldots, \xi_{q+q'}) = \\
& (-1)^{qp'}\Vc\Big(a\ns{0}\om \ot \eta\ns{0}\wg \zeta \ot \widetilde{f} \ot S(a\ns{1})\eta\ns{1}\cdot \widetilde{g}\Big)( X_1,\ldots,X_{p+p'}\mid \xi_1,\ldots, \xi_{q+q'}) = \\
& (-1)^{qp'}  \Big\langle \eta\ns{0}\wg \zeta,  X_1,\ldots,X_{p+p'}\Big\rangle \sum_{\s \in S_{q+q'}} (-1)^\s \Big\langle \widetilde{f} \ot S(a\ns{1})\eta\ns{-1}\cdot \widetilde{g}, \xi_{\s(1)},\ldots, \xi_{\s(q+q')}\Big\rangle a\ns{0}\om = \\
& (-1)^{qp'}  \Big\langle \eta\wg \zeta,  X_1,\ldots,X_{p+p'}\Big\rangle \sum_{\s \in S_{q+q'}} (-1)^\s \Big\langle \widetilde{f} \ot \widetilde{g}, \xi_{\s(1)},\ldots, \xi_{\s(q+q')}\Big\rangle a\om,
\end{align*}
where on the last equality we used the fact that the Lie algebra elements are primitive, and that the (non-identity) elements of $\Fc(N)$ are zero under the counit (when evaluated on the identity). Employing the anti-symmetrization map $\a:C^n(\Fc(N),V) \to C^n(\Fn,V)$, see for instance \cite[Subsect. 4.1]{RangSutl}, and setting $\widetilde{\a(f)} := \a(f^1)\wdots \a(f^q)$ corresponding to $\widetilde{f}:=f^1\odots f^q$, we may rewrite the cup product as
\begin{equation*}
\Vc\Big((a\ot \eta \ot \widetilde{f}) \ast (\om\ot \zeta \ot \widetilde{g})\Big) = (-1)^{qp'} a\om \ot  \eta\wg \zeta \ot \widetilde{\a(f)} \wg \widetilde{\a(g)}.
\end{equation*}
The claim now follows from \eqref{cup-total}.
\end{proof}

\ni A few words on the above results are in order. We recall that the multiplicative generators of the cohomology on the range are already known, see \cite[Thm. 2.2.7]{Fuks-book}, and the Subsection \ref{Lie-cohom-coeff} above. In addition, it is shown in Theorem \ref{thm-char-homm} that the van Est type map \eqref{VE} is an isomorphism on the level of the cohomologies. Hence, by Theorem \ref{VE-multp} the (multiplicative) generators of the Gelfand-Fuks cohomology (which are the characteristic classes of foliations) can be pulled back to the Hopf-cyclic cohomology of $\Hc_n$. More explicitly, $\Vc^{-1}(\lambda_k) \in C^{2k-1}(\Fc(N) \acl U(\Fs), \Om_{n\d}^{\leq1})$ and $\Vc^{-1}(\mu) \in C^1(\Fc(N) \acl U(\Fs), \Om_{n\d}^{\leq1})$, subject to the relations \eqref{Vey-basis-cond}, form a basis for the Hopf-cyclic cohomology of $\Hc_n$, with coefficients in $\Om_{n\d}^{\leq1}$. In the next subsection we shall illustrate this pull-back procedure for $n=1$, and demonstrate the inverse images under \eqref{VE} of the classes \eqref{lambda} and \eqref{mu}.

\subsection{The Hopf-cyclic classes}

We illustrate the transfer of classes in the case of $n=1$. For the ease of the presentation we are going to work with the representatives in the completion of the Lie algebra $W_1$ with respect to the natural topology (strict inductive limit topology of \cite{BonnFlatGersPinc94,BonnSter05}), and of the Hopf algebra $\Hc_1$, and of the projected tensor product $\projot$ (to which we shall keep referring as $\ot$). For convenience, we refer the reader to \cite{RangSutl-V} for the Hopf-cyclic cohomology for the topological Hopf-algebras.

\ni Let us first note that we shall adopt the basis $\{e_i\mid i \geq -1\}$ of the Lie algebra $W_1$, \cite[Subsect. 1.1.2]{Fuks-book}, and the basis $\{f^i \mid i\geq 0\}$ of the $W_1$-module $\Om_1^1$, \cite[Sect. 5.3]{Antal-thesis}, where the $W_1$-action is given by 
\begin{equation*}
e_i \cdot f^j = (i+j+1) f^{i+j}.
\end{equation*}

\ni As it is noted, the cohomology $H^\ast(W_1,\Om_1^{\leq1})$ is generated by the classes \eqref{lambda} and \eqref{mu}. More explicitly, if 
\begin{eqnarray*}
\xi & = &c_{-1} e_{-1} + c_0 e_0 + c_1 e_1 + \ldots \\
& = & c_{-1} \frac{\part}{\part x} + c_0 x\frac{\part}{\part x} + c_1 x^2 \frac{\part}{\part x} + \ldots 
\end{eqnarray*}
one has
\begin{equation*}
\lambda(\xi) = c_0 + 2c_1x + 3c_2x^2 + \ldots
\end{equation*}
that is, setting $\{\t^i \mid i\geq -1\}$ such that $\langle \t^i,\, e_j \rangle = \d^i_j$,
\begin{equation}\label{first-cocycle}
\lambda = \one \ot \t^0 + \sum_{i\geq 1}\,(i+1)x^i \ot \t^i \in C^{0,1}(\Om_1^{\leq1}, \Fs^\ast,\Fn^\ast) \oplus C^{1,0}(\Om_1^{\leq1}, \Fs^\ast,\Fn^\ast),
\end{equation}
and similarly 
\begin{equation}\label{second-cocycle}
\mu = \sum_{i\geq 1}\,(i+1)if^{i-1} \ot \t^i \in C^{1,0}(\Om_1^{\leq1}, \Fs^\ast,\Fn^\ast)
\end{equation}
on the bicomplex \eqref{g-1-g-2-bicomplex}. We note also that
\begin{align}\label{horizontal}
\begin{split}
& \hD_{\rm CE}(\mu)(e_p,e_q) = \mu([e_p,e_q]) - e_p\cdot \mu(e_q) + e_q\cdot \mu(e_p) = \\
& (q-p) \mu(e_{p+q}) - e_p\cdot \mu(e_q) + e_q\cdot \mu(e_p) = \\
& (q-p) (p+q+1)(p+q)f^{p+q-1} - e_p\cdot (q+1)qf^{q-1} + e_q\cdot (p+1)pf^{p-1} = \\
& (q-p) (p+q+1)(p+q)f^{p+q-1} - (q+1)q (p+q) f^{p+q} + (p+1)p(p+q)f^{p+q} = 0,
\end{split}
\end{align}
as well as,
\begin{align}\label{vertical-2}
\begin{split}
& \vDD_{\rm CE}(\mu) = \mu \cdot e_{-1} \ot \t^{-1} + \mu \cdot e_0 \ot \t^0 = \\
& \sum_{i\geq 1}\,(i+1)if^{i-1}\cdot e_{-1} \ot \t^{-1} \ot \t^i + \sum_{i\geq 1}\,(i+1)if^{i-1} \ot \t^{-1} \ot \t^i\cdot e_{-1}  + \\
& \sum_{i\geq 1}\,(i+1)if^{i-1}\cdot e_0 \ot \t^0 \ot \t^i + \sum_{i\geq 1}\,(i+1)if^{i-1} \ot \t^0 \ot \t^i\cdot e_0 = \\
& - \sum_{i\geq 2}\,(i+1)i(i-1) f^{i-2} \ot \t^{-1} \ot \t^i + \sum_{i\geq 1}\,(i+2)(i+1)if^{i-1} \ot \t^{-1} \ot \t^{i+1}  + \\
& - \sum_{i\geq 1}\,(i+1)i^2 f^{i-1} \ot \t^0 \ot \t^i + \sum_{i\geq 1}\,(i+1)i^2f^{i-1} \ot \t^0 \ot \t^i = 0.
\end{split}
\end{align}
As for $\lambda \in C^{0,1}(\Om_1^{\leq1}, \Fs^\ast,\Fn^\ast) \oplus C^{1,0}(\Om_1^{\leq1}, \Fs^\ast,\Fn^\ast)$, we observe that
\begin{equation}\label{vertical-1}
\vDD_{\rm CE}(\one \ot \t^0) = \one\cdot e_{-1} \ot \t^{-1}\wg \t^0 + \one\cdot e_0 \ot \t^0\wg \t^0 + \one \ot d_{\rm DR}(\t^0) = 0.
\end{equation}
On the other hand,
\begin{align*}
\hD_{\rm CE}(\one \ot \t^0)(e_p) = e_p \cdot (\one \ot \t^0) = - \one \ot \t^0 \cdot e_p = \begin{cases}
2 (\one \ot \t^{-1}), & \text{\rm if}\,\, p=1, \\
0, & \text{\rm if} \,\, p> 1,
\end{cases}
\end{align*}
that is,
\begin{equation*}
\hD_{\rm CE}(\one \ot \t^0) = 2(\one \ot \t^{-1} \ot \t^1),
\end{equation*}
and
\begin{align*}
& \vDD_{\rm CE}(\sum_{i\geq 1}\,(i+1)x^i \ot \t^i) = \\
& \hspace{2cm}\sum_{i\geq 1}\,(i+1)x^i\cdot e_{-1} \ot \t^{-1} \ot \t^i  + \sum_{i\geq 1}\,(i+1)x^i \cdot e_0 \ot \t^0\ot \t^i + \\
& \sum_{i\geq 1}\,(i+1)x^i \ot \t^{-1}\ot \t^i\cdot e_{-1} + \sum_{i\geq 1}\,(i+1)x^i  \ot \t^0\ot \t^i\cdot e_0 = \\
& - \sum_{i\geq 1}\,(i+1)i x^{i-1} \ot \t^{-1} \ot \t^i  - \sum_{i\geq 1}\,(i+1)i x^i \ot \t^0\ot \t^i + \\
& \sum_{i\geq 1}\,(i+2)(i+1)x^i \ot \t^{-1}\ot \t^{i+1} + \sum_{i\geq 1}\,(i+1)ix^i  \ot \t^0\ot \t^i = \\
& - 2(\one \ot \t^{-1} \ot \t^1) - \sum_{i\geq 2}\,(i+1)i x^{i-1} \ot \t^{-1} \ot \t^i  - \sum_{i\geq 1}\,(i+1)i x^i \ot \t^0\ot \t^i + \\
& \sum_{i\geq 1}\,(i+2)(i+1)x^i \ot \t^{-1}\ot \t^{i+1} + \sum_{i\geq 1}\,(i+1)ix^i  \ot \t^0\ot \t^i = - 2(\one \ot \t^{-1} \ot \t^1).
\end{align*}
Finally,
\begin{align*}
& \hD_{\rm CE}(\sum_{i\geq 1}\,(i+1)x^i \ot \t^i)(e_p,e_q) = \\
& \sum_{i\geq 1}\,(i+1)x^i \t^i([e_p,e_q]) - e_p \cdot \sum_{i\geq 1}\,(i+1)x^i \t^i(e_q) + e_q \cdot \sum_{i\geq 1}\,(i+1)x^i \t^i(e_p) = \\
& (q-p)(p+q+1)x^{p+q} - (q+1)(p+q+1)x^{p+q} + (p+1)(p+q+1)x^{p+q} = 0.
\end{align*}

\ni Referring the reader to \cite{McCleary-book} for details on spectral sequences, we now investigate the generators of the cohomology $H^\ast(W_1,\Om_1^{\leq1})$ in the 1st page of the spectral sequence associated to the natural filtration of the bicomplex \eqref{g-1-g-2-bicomplex}.

\begin{proposition}
On the $E_1$-term of the spectral sequence corresponding to the natural filtration of the bicomplex \eqref{g-1-g-2-bicomplex}, we have
\begin{equation*}
[\one \ot \t^0]_1 \in E_1^{0,1}, \hspace{3cm} [\sum_{i\geq 1}\,(i+1)if^{i-1} \ot \t^i]_1 \in E_1^{1,0}.
\end{equation*}
\end{proposition}

\begin{proof}
The 0-page $(E_0,d_0)$ of the spectral sequence consists of the vertical cohomology classes. As a result, we conclude from \eqref{vertical-1} and \eqref{vertical-2} that 
\begin{equation*}
[\one \ot \t^0]_0 \in E_0^{0,1}, \hspace{3cm} [\sum_{i\geq 1}\,(i+1)if^{i-1} \ot \t^i]_0 \in E_0^{1,0}.
\end{equation*}
On the $E_1$-level we have $d_1:E_1^{p,q} \to E_1^{p+1,q}$, that is, horizontal coboundary map acting on the vertical cohomology classes. We then note that
\begin{equation*}
d_1[\one \ot \t^0]_1 = [\hD_{\rm CE}(\one \ot \t^0)]_1 = [-\vDD_{\rm CE}(\sum_{i\geq 1}\,(i+1)x^i \ot \t^i)]_1 = [0]_1,
\end{equation*}
hence
\begin{equation*}
[\one \ot \t^0]_1 \in E_1^{0,1}.
\end{equation*}
On the other hand, \eqref{horizontal} implies that 
\begin{equation*}
[\sum_{i\geq 1}\,(i+1)if^{i-1} \ot \t^i]_1 \in E_1^{1,0}.
\end{equation*}
\end{proof}

\ni We now pull these two classes back to the Hopf-cyclic bicomplex \eqref{UF}. Once again we recall from \cite[Subsect. 4.2]{Antal-thesis} the affine coordinates $\{{\bf x}_i\mid i\geq 1\}$ of $N$, which are given by 
\begin{equation}\label{affine-basis}
{\bf x}_i(e_J) = \begin{cases}
1, & \text{\rm if}\,\, J = (i), \\
0, & \text{\rm otherwise},
\end{cases}
\end{equation}
where $e_J:= e_{j_1}\ldots e_{j_n}$ for $J=(j_1,\ldots, j_n)$.

\begin{proposition}
On the $E_1$-term of the spectral sequence corresponding to the natural filtration of the bicomplex \eqref{UF+*}, we have
\begin{equation*}
[\one \ot \t^0]_1 \in E_1^{0,1}, \hspace{1cm} [\sum_{i\geq 1}\,(i+1)if^{i-1} \ot {\bf x}_i]_1 \in E_1^{1,0}.
\end{equation*}
\end{proposition}

\begin{proof}
We have by \cite[Thm. 4.10]{RangSutl} that the characteristic homomorphism \eqref{VE} is an isomorphism on the $E_1$-level of the spectral sequences associated to the natural filtrations of the bicomplexes \eqref{g-1-g-2-bicomplex} and \eqref{UF+*}. It already follows from \eqref{vertical-2} that 
\begin{equation*}
[\sum_{i\geq 1}\,(i+1)if^{i-1} \ot {\bf x}_i]_0 \in E_0^{1,0},
\end{equation*}
and since \eqref{VE} is an isomorphism of $E_1$-terms,
\[
b_N^\ast\left(\sum_{i\geq 1}\,(i+1)if^{i-1} \ot {\bf x}_i\right) \in E_1^{2,0}
\]
is a vertical coboundary. But then, since it resides in the 0th row, we conclude that
\begin{equation*}
b_N^\ast(\sum_{i\geq 1}\,(i+1)if^{i-1} \ot {\bf x}_i) = 0.
\end{equation*}
Furthermore, in view of \cite[Mapping Lemma 5.2.4]{Weibel-book}, \eqref{VE} is also a map of $E_r$-terms as well, for any $r\geq 1$. Thus, from
\[
b_N^\ast(\one \ot \t^0) = - 2(\one \ot \t^{-1} \ot {\bf x}_1) = d_{\rm CE}\left(\sum_{i\geq 1}\,(i+1)x^i \ot {\bf x}_i\right),
\]
we conclude that
\[
b_N^\ast\left(\sum_{i\geq 1}\,(i+1)x^i \ot {\bf x}_i\right) = 0.
\]
\end{proof}

\begin{corollary}
The total cohomology of the bicomplex \eqref{UF+*} is generated by the classes
\begin{equation}\label{cocycle-lambda}
\lambda':=\one \ot \t^0 \oplus \sum_{i\geq 1}\,(i+1)x^i \ot {\bf x}_i \in C^{0,1}(\Om_{1\d}^{\leq1}, \Fs^\ast,\Fc(N)) \oplus C^{1,0}(\Om_{1\d}^{\leq1}, \Fs^\ast,\Fc(N))
\end{equation}
and
\begin{equation}\label{cocycle-mu}
\mu':=\sum_{i\geq 1}\,(i+1)f^i \ot {\bf x}_i \in C^{1,0}(\Om_{1\d}^{\leq1}, \Fs^\ast,\Fc(N)).
\end{equation}
\end{corollary}

\ni Applying the Poincar\'e duality \eqref{Poincare-duality}, we push the above classes to
\[
\one \ot e_{-1} \oplus \sum_{i\geq 1}\,(i+1)x^i \ot e_{-1}\wg e_0 \,\ot\, {\bf x}_i \in C^{0,1}(\Om_{1\d}^{\leq1}, \Fs,\Fc(N)) \oplus C^{1,2}(\Om_{1\d}^{\leq1}, \Fs,\Fc(N)),
\]
and
\[
\sum_{i\geq 1}\,(i+1)f^i \ot e_{-1}\wg e_0 \ot {\bf x}_i \in C^{1,2}(\Om_{1\d}^{\leq1}, \Fs,\Fc(N)).
\]
We then observe that 
\[
\p_{\rm CE}(\one \ot e_{-1}\wg e_0) = \one \ot [e_{-1},\, e_0] = \one \ot e_{-1}
\]
and
\[
b_N(\one \ot e_{-1}\wg e_0) = 0.
\]
Hence the former class is cohomologous to
\[
\sum_{i\geq 1}\,(i+1)x^i \ot e_{-1}\wg e_0 \,\ot\, {\bf x}_i \in \oplus C^{1,2}(\Om_{1\d}^{\leq1}, \Fs,\Fc(N)).
\]
On the next move, we apply the anti-symmetrization map \eqref{antsym1} to get
\[
\frac{1}{2}\,\sum_{i\geq 1}\,(i+1)x^i \ot (e_{-1}\ot e_0 - e_0\ot e_{-1}) \ot {\bf x}_i \in C^{1,2}(\Om_{1\d}^{\leq1}, U(\Fs),\Fc(N))
\]
and
\[
\frac{1}{2}\sum_{i\geq 1}\,(i+1)f^i \ot (e_{-1}\ot e_0 - e_0\ot e_{-1}) \ot {\bf x}_i \in C^{1,2}(\Om_{1\d}^{\leq1}, U(\Fs),\Fc(N)).
\]
We next carry the classes from the total complex to the diagonal subcomplex via the Alexander-Whitney map, see for instance \cite{KustRognTuse02}. This way we obtain
\begin{align*}
& \frac{1}{2}\sum_{i\geq 1}\,(i+1)x^i \ot (e_{-1}\ot e_0 \ot 1 - e_0\ot e_{-1} \ot 1) \ot 1 \ot 1 \ot {\bf x}_i \in \\
& \hspace{8cm} {\rm Diag}^3(U(\Fs),\Fc(N),\Om_{1\d}^{\leq1}),
\end{align*}
and
\begin{align*}
& \frac{1}{2}\sum_{i\geq 1}\,(i+1)f^i \ot (e_{-1}\ot e_0 \ot 1 - e_0\ot e_{-1} \ot 1) \ot 1 \ot 1 \ot {\bf x}_i \in \\
& \hspace{8cm} {\rm Diag}^3(U(\Fs),\Fc(N),\Om_{1\d}^{\leq1}).
\end{align*}
Finally, we apply \eqref{PSI-1} to get the Hopf-cyclic representatives of the classes \eqref{lambda} and \eqref{mu}. As a result, we conclude the following.

\begin{corollary}
The classes \eqref{lambda} and \eqref{mu} are represented, in the Hopf-cyclic cohomology of the Hopf algebra $\Hc_1$ with coefficients in $\Om^{\leq1}_{1\d}$, by the 3-cocycles given by
\begin{align*}
& \lambda_{Hopf} := \frac{1}{2}\sum_{i\geq 0}\,(i+1)x^i \ot e_{-1} \ot e_0 \ot {\bf x}_i - \\
&  \sum_{i\geq 0}\,(i+1)x^i \ot e_0 \ot {\bf x}_1e_0 \ot {\bf x}_i -   \frac{1}{2}\sum_{i\geq 0}\,(i+1)x^i \ot e_0 \ot e_{-1} \ot {\bf x}_i,
\end{align*}
and
\begin{align*}
& \mu_{Hopf} := \frac{1}{2}\sum_{i\geq 0}\,(i+1)f^i \ot e_{-1} \ot e_0 \ot {\bf x}_i - \\
&  \sum_{i\geq 0}\,(i+1)f^i \ot e_0 \ot {\bf x}_1e_0 \ot {\bf x}_i -  \frac{1}{2}\sum_{i\geq 0}\,(i+1)f^i \ot e_0 \ot e_{-1} \ot {\bf x}_i.
\end{align*}
\end{corollary}

\begin{proof}
The claim follows from
\begin{align*}
& \lambda_{Hopf} = \Psi_{\acl}\left( \frac{1}{2}\sum_{i\geq 1}\,(i+1)x^i \ot (e_{-1}\ot e_0 \ot 1 - e_0\ot e_{-1} \ot 1) \ot 1 \ot 1 \ot {\bf x}_i \right) = \\
&  \frac{1}{2}\sum_{i\geq 1}\,(i+1)x^i \ot (1 \acl e_{-1}\ns{0}) \ot (e_{-1}\ns{1} \acl e_0\ns{0}) \ot ({\bf x}_ie_{-1}\ns{2}e_0\ns{1} \acl 1) - \\
&  \frac{1}{2}\sum_{i\geq 1}\,(i+1)x^i \ot (1 \acl e_0\ns{0}) \ot (e_0\ns{1} \acl e_{-1}\ns{0}) \ot ({\bf x}_ie_0\ns{2}e_{-1}\ns{1} \acl 1),
\end{align*}
and the similar arguments for $\mu_{Hopf} \in C^3(\Hc_1,\Om^{\leq1}_{1\d})$.
\end{proof}

\subsection{Connection with the group cohomology}

We shall now construct more compact representatives of the cocycles \eqref{cocycle-lambda} and \eqref{cocycle-mu} in the group cohomology of the group $N$. Let us consider the bigraded space 
\begin{equation}\label{group-bicomplex}
C_{\rm pol}^{\ast,\ast}(N,\Fs,\Om_1^{\leq1}) := \bigoplus_{p,q\geq 0}C_{\rm pol}^{p,q}(N,\Fs,\Om_1^{\leq1}), \qquad C_{\rm pol}^{p,q}(N,\Fs,\Om_1^{\leq1}):=C_{\rm pol}^q(N, \Om_1^{\leq1}\ot \wdg^p \Fs^*)
\end{equation}
where $C_{\rm pol}^q(N, \Om^{\leq1}_{1\d}\ot \wdg^p\Fs^*)$ refers to the space of $q$-cochains of the group cohomology of $N$, with coefficients in the $N$-module $\Om^{\leq1}_{1\d}\ot \wdg^p\Fs^*$, see for instance \cite[Sect. 2]{HochSerr53-II}. Namely, the set of (homogeneous) polynomial cochains
\begin{equation*}
\phi:\underset{(q+1)-many}{\underbrace{N\times \cdots \times N}}\lra \Om_1^{\leq1}\ot \wdg^p\Fs^*,
\end{equation*}
satisfying
\[
\phi(\psi\psi_0, \ldots, \psi\psi_q) = \psi\cdot \phi(\psi_0,\ldots, \psi_q),
\]
together with the coboundary
 \begin{align}\label{b_N}
 \begin{split}
& b_N:C_{\rm pol}^q(N,\Om_1^{\leq1}\ot \wdg^p\Fs^*)\lra C_{\rm pol}^{q+1}(N,\Om_1^{\leq1}\ot \wdg^p\Fs^*),\\
 &b_N(\phi)(\psi_0,\cdots,\psi_{q+1})= \sum_{i=0}^{q+1}\,(-1)^i\,\phi(\psi_0, \ldots,\widehat{\psi}_i,\ldots,\psi_{q+1}).
\end{split}
 \end{align}
 The action of the group $N$ is given explicitly by
 \[
 f(x)dx \cdot \psi := f(\psi(x))\psi'(x)dx,
 \]
 see \cite[Sect. 1]{OvsiRoge98}. In addition, we introduce the coboundary
\begin{align}\label{b_s}
\begin{split}
&b_\Fs:C_{\rm pol}^q(N,\Om_1^{\leq1}\ot \wdg^p\Fs^*)\lra C_{\rm pol}^{q}(N,\Om_1^{\leq1}\ot \wdg^{p+1}\Fs^*),\\
&b_\Fs(\phi)(\psi_0,\cdots,\psi_q):=d^\Om_{\rm CE}(\phi(\psi_0,\cdots,\psi_q)) - \sum_{j=-1}^0\t^j \wg (e_j \rt \phi)(\psi_0,\cdots,\psi_q),
\end{split}
\end{align}
where $d^\Om_{\rm CE}:\Om_1^{\leq1}\ot\wg^p\,\Fs^\ast \to \Om_1^{\leq1}\ot\wg^{p+1}\,\Fs^\ast$ is the Lie algebra cohomology coboundary (with coefficients in $\Om_1^{\leq1}$), and for any $X\in \Fs$ and $\phi \in C^q(N,\Om_1^{\leq1}\ot \wdg^p\Fs^*)$,
\[
(X\rt \phi)(\psi_1,\ldots, \psi_q) := \sum_{j=0}^q \left.\frac{d}{dt}\right|_{t=0}\phi(\psi_0,\ldots, \psi_j\lt \exp(tX),\ldots, \psi_q).
\]
We thus have the following.

\begin{proposition}
The coboundaries \eqref{b_N} and $b_s$ commute, that is,
\[
b_N \circ b_\Fs = b_\Fs \circ b_N.
\]
\end{proposition}

\begin{proof}
On one hand we have
\begin{align*}
& (b_N \circ b_\Fs) (\phi) (\psi_0,\ldots, \psi_{q+1}) = b_N(b_\Fs(\phi))  (\psi_0,\ldots, \psi_{q+1}) = \\
& \sum_{i=0}^{q+1}\,(-1)^i\,b_\Fs(\phi)(\psi_0, \ldots, \widehat{\psi}_i,\ldots, \psi_{q+1}) =\\
& \sum_{i=0}^{q+1}\,(-1)^i\,\left(d^\Om_{\rm CE}(\phi(\psi_0, \ldots, \widehat{\psi}_i,\ldots, \psi_{q+1})) - 
\t^j \wg (e_j \rt \phi)(\psi_0, \ldots, \widehat{\psi}_i,\ldots, \psi_{q+1})\right),
\end{align*}
and on the other hand,
\begin{align*}
& (b_\Fs \circ b_N) (\phi) (\psi_0,\ldots, \psi_{q+1}) = b_\Fs(b_N(\phi))  (\psi_0,\ldots, \psi_{q+1}) = \\
& d^\Om_{\rm CE} (b_N(\phi)(\psi_0,\ldots, \psi_{q+1})) - 
\t^j \wg (e_j \rt b_N(\phi))(\psi_0,\cdots,\psi_q) = \\
& d^\Om_{\rm CE} \left(\sum_{i=0}^{q+1}\,(-1)^i\,\phi(\psi_0, \ldots, \widehat{\psi}_i,\ldots, \psi_{q+1})\right) - 
\sum_{i=0}^{q+1}\,(-1)^i\,\t^j \wg (e_j \rt \phi)(\psi_0, \ldots, \widehat{\psi}_i,\ldots, \psi_{q+1}).
\end{align*}
\end{proof}

\ni As a result, we arrive at the bicomplex
\begin{align*}
\begin{xy} \xymatrix{  \vdots & 
 &\vdots &&\vdots\\
 \Om_1^{\leq1}\ot \wdg^2\Fs^\ast  \ar[u]^{b_\Fs} \ar[rr]^{b_N\,\,\,\,\,\,\,\,\,\,\;\;\;\;\;}& &  C_{\rm pol}^1(N,\Om_1^{\leq1}\ot \wdg^2\Fs^\ast)\ar[u]^{b_\Fs} \ar[rr]^{b_N\,\,\,\,\,\,\,\,\,\,}& & C_{\rm pol}^2(N,\Om_1^{\leq1}\ot \wdg^2\Fs^\ast)\ar[u]^{b_\Fs} \ar[r]^{\;\;\;\;\;\;\;\;\;\;\;\;\;\,\,\,\,b_N} & \hdots &  \\
 \Om_1^{\leq1}\ot \Fs^\ast  \ar[u]^{b_\Fs}\ar[rr]^{b_N ~~~~~}& & C_{\rm pol}^1(N,\Om_1^{\leq1}\ot \Fs^\ast) \ar[u]^{b_\Fs} \ar[rr]^{b_N}&& C_{\rm pol}^2(N,\Om_1^{\leq1}\ot  \Fs^\ast)  \ar[u]^{b_\Fs} \ar[r]^{\;\;\;\;\;\;\;\;\;\;b_N }& \hdots&  \\
   \Om_1^{\leq1}\ar[u]^{b_\Fs}\ar[rr]^{b_N~~~~~~~}& & C_{\rm pol}^1(N,\Om_1^{\leq1})\ar[u]^{b_\Fs}\ar[rr]^{b_N }& &C_{\rm pol}^2(N,\Om_1^{\leq1}) \ar[u]^{b_\Fs} \ar[r]^{\;\;\;\;\;b_N} & \hdots& }
\end{xy}
\end{align*}
The bicomplex \eqref{group-bicomplex} is evidently a sub-bicomplex of
\begin{equation}\label{cont-group-bicomplex}
C_{\rm cont}^{\ast,\ast}(N,\Fs,\Om_1^{\leq1}) := \bigoplus_{p,q\geq 0}C_{\rm cont}^{p,q}(N,\Fs,\Om_1^{\leq1}), \qquad C_{\rm cont}^{p,q}(N,\Fs,\Om_1^{\leq1}):=C_{\rm cont}^q(N, \Om_1^{\leq1}\ot \wdg^p \Fs^*)
\end{equation}
of continuous group cochains. Furthermore, $C_{\rm cont}^q(N,\Om^1_{1\d}\ot \wdg^p\Fs^*)$ may be considered as the set of continuous (inhomogeneous cochains)
\begin{equation*}
\wbar{\phi}:\underset{q-many}{\underbrace{N\times \cdots \times N}}\lra \Om_1^{\leq1}\ot \wdg^p\Fs^*,
\end{equation*}
via the identification $\wbar{\phi}(\psi_1,\ldots,\psi_q) = \phi(\psi_1\ldots\psi_q,\psi_2\ldots\psi_q,\ldots, \psi_q,e)$. This way, the horizontal coboundary transforms into 
 \begin{align*}
& \wbar{b}_N:C^q(N,\Om_1^{\leq1}\ot \wdg^p\Fs^*)\lra C^{q+1}(N,\Om_1^{\leq1}\ot \wdg^p\Fs^*),\\
 &\wbar{b}_N(\wbar{\phi})(\psi_1,\cdots,\psi_{q+1}) = \\
 & \wbar{\phi}(\psi_2,\ldots,\psi_{q+1}) + \sum_{i=1}^q\,(-1)^i\,\wbar{\phi}(\psi_1,\ldots,\psi_{i}\psi_{i+1},\ldots,\psi_{q+1})+ (-1)^{q+1} \wbar{\phi}(\psi_1,\cdots,\psi_{q}) \cdot \psi_{q+1}.
 \end{align*}

\begin{proposition}\label{prop-van-est-diff}
We have the van Est-type isomorphism
\[
H_{\rm cont}^\ast(\Diff(\Rb), \Om_1^{\leq1}) \cong H^\ast(W_1, \Om_1^{\leq1}).
\]
\end{proposition}

\begin{proof}
In view of the van Est isomorphism \cite[Prop. 1.5]{Dupo76} on the vertical level, we note from \cite[Lemma 1]{HochSerr53-II} that the $E^1$-term of the spectral sequence, associated to the natural filtration of the bicomplex \eqref{cont-group-bicomplex}, is identified with the $E^1$-term of the Cartan-Leray spectral sequence  which computes, by \cite[Prop. 5]{HochSerr53-II}, the group cohomology $H_{\rm cont}^\ast(\Diff(\Rb), \Om_1^{\leq1})$, regarding the decomposition $\Diff(\Rb)=S\cdot N$.


\ni The claim now follows from the identification of the bicomplex \eqref{cont-group-bicomplex} with the Lie algebra cohomology bicomplex \eqref{g-1-g-2-bicomplex}, which requires the commutativity of the inverse limit and the cohomology as follows. The (profinite) group $N$ is given as an inverse limit, see \cite[Eqn. (1.52)]{MoscRang11}:
\[
N = \underset{k\to \infty}{\underset{\lla}{\lim}}\,N_k.
\] 
We have projections $\pi_{ij}:N_i\to N_j$, from the group $N_i$ of invertible $i$-jets at $0\in \Rb$ to the group $N_j$ of invertible $j$-jets at $0\in \Rb$,  for any $i \geq j$, see for instance \cite[Sect. IV.13]{KolarMichorSlovak-book}. Therefore, the inverse system $(N_i, \pi_{ij})$ satisfies the Mittag-Leffler condition, \cite[Thm. 1]{Dimi04}. Hence
\begin{align*}
& H_{\rm cont}^\ast(N, \Om_1^{\leq1}\ot \Fs^\ast) = H_{\rm cont}^\ast\left(\underset{k\to \infty}{\underset{\lla}{\lim}}\,N_k, \Om_1^{\leq1}\ot \Fs^\ast\right) =  \underset{k\to \infty}{\underset{\lla}{\lim}}\,H_{\rm cont}^\ast\left(N_k, \Om_1^{\leq1}\ot \Fs^\ast\right).
\end{align*}
On the other hand, $\Fn_k:= \langle \left\{e_j\mid  j \geq k \right\}\rangle$ being the Lie algebra of the group $N_k$, \cite[Sect. IV.13]{KolarMichorSlovak-book}, on the infinitesimal level we have the inverse system $(\Fn_i,\pi_{ij})$ of Lie algebras with the projections $\pi_{ij}:\Fn_i \to \Fn_j$ for any $i \geq j$. As such, 
\begin{align*}
& \underset{k\to \infty}{\underset{\lla}{\lim}}\,H_{\rm cont}^\ast\left(N_k, \Om_1^{\leq1}\ot \Fs^\ast\right) \cong \underset{k\to \infty}{\underset{\lla}{\lim}}\,H_{\rm cont}^\ast\left(\Fn_k, \Om_1^{\leq1}\ot \Fs^\ast\right) = \\
& H_{\rm cont}^\ast\left(\underset{k\to \infty}{\underset{\lla}{\lim}}\,\Fn_k, \Om_1^{\leq1}\ot \Fs^\ast\right) = H^\ast(\Fn, \Om_1^{\leq1}\ot \Fs^\ast),
\end{align*}
where we note for the first (van Est) isomorphism that the maximal compact subgroup of $N_k$, for any $k\geq 1$, is $SO(1)$. We refer the reader to \cite[Thm. 2.4\, \&\,Thm. 2.5]{Dart94} for further identifications of these cohomologies. 
\end{proof}

\ni On the next step, let us consider the (coinvariant) bigraded space
\begin{equation}\label{coinv-bicomplex}
C^{\ast,\ast}_{\rm coinv}(\Om_{n\d}^{\leq1},\Fs^\ast,\Fc(N)) = \bigoplus_{p,q\geq 0} C^{p,q}_{\rm coinv}(\Om_{n\d}^{\leq1},\Fs^\ast,\Fc(N)),
\end{equation}
where
\begin{align*}
& C^{p,q}_{\rm coinv}(\Om_{n\d}^{\leq1},\Fs^\ast,\Fc(N)) := \left(\Om_{n\d}^{\leq1} \ot \wg^p\Fs^\ast \ot \Fc(N)^{\ot\,q+1}\right)^{\Fc(N)} \\
& = \left\{ v \ot \eta \ot \widetilde{f} \mid v\ns{0} \ot \eta\ns{0} \ot \widetilde{f} \ot S(v\ns{1}\eta\ns{1}) = v \ot \eta \ot \widetilde{f}\ns{0} \ot \widetilde{f}\ns{1} \right\},
\end{align*}
and
\[
\widetilde{f}\ns{0} \ot \widetilde{f}\ns{1} := f^0\ps{1} \odots f^q\ps{1} \ot f^0\ps{2} \ldots f^q\ps{2}.
\]
As in \cite[Prop. 1.15]{MoscRang11}, \eqref{coinv-bicomplex} can be identified with the bicomplex \eqref{bicomplex-Om-s-F}.

\begin{proposition}
The mapping
\[
\Ic: C^{p,q}(\Om_{n\d}^{\leq1},\Fs^\ast,\Fc(N)) \lra C^{p,q}_{\rm coinv}(\Om_{n\d}^{\leq1},\Fs^\ast,\Fc(N))
\]
given by
\begin{align*}
& \Ic(v \ot \eta \ot f^1\odots f^q):= \\
& \hspace{1cm} v\ns{0} \ot \eta\ns{0} \ot f^1\ps{1} \ot S(f^1\ps{2})f^2\ps{1} \odots S(f^{q-1}\ps{2})f^q\ps{1} \ot S(v\ns{1}\eta\ns{1}f^q\ps{2})
\end{align*}
is an isomorphism.
\end{proposition}

\begin{proof}
We first show that the image is indeed in the coinvariant bicomplex. To this end we note that
\begin{align*}
& \left(v\ns{0} \ot \eta\ns{0}\right)\ns{0} \ot f^1\ps{1} \ot S(f^1\ps{2})f^2\ps{1} \ot \cdots \\
& \hspace{1cm}\cdots \ot S(f^{q-1}\ps{2})f^q\ps{1} \ot S(v\ns{1}\eta\ns{1}f^q\ps{2}) \ot S(\left(v\ns{0} \ot \eta\ns{0}\right)\ns{1}) = \\
& v\ns{0} \ot \eta\ns{0} \ot f^1\ps{1} \ot S(f^1\ps{2})f^2\ps{1} \ot\cdots \\
& \hspace{1cm}\cdots \ot S(f^{q-1}\ps{2})f^q\ps{1} \ot S(v\ns{2}\eta\ns{2}f^q\ps{2}) \ot S(v\ns{1} \eta\ns{1}) = \\
& v\ns{0} \ot \eta\ns{0} \ot f^1\ps{1} \ot S(f^1\ps{4})f^2\ps{1} \odots S(f^{q-1}\ps{4})f^q\ps{1} \ot S(v\ns{2}\eta\ns{2}f^q\ps{4}) \ot \\
& f^1\ps{2} S(f^1\ps{3})f^2\ps{2} \ldots S(f^{q-1}\ps{3})f^q\ps{2}  S(v\ns{1}\eta\ns{1}f^q\ps{3}) = \\
& v\ns{0} \ot \eta\ns{0} \ot \left(f^1\ps{1}\right)\ps{1} \ot \left(S(f^1\ps{2})f^2\ps{1}\right)\ps{1} \ot\cdots \\
& \hspace{2cm} \cdots \ot \left(S(f^{q-1}\ps{2})f^q\ps{1}\right)\ps{1} \ot S(v\ns{1}\eta\ns{1}f^q\ps{2})\ps{1} \ot \\
& \left(f^1\ps{1}\right)\ps{2} \left(S(f^1\ps{2})f^2\ps{1}\right)\ps{2} \ldots \left(S(f^{q-1}\ps{2})f^q\ps{1}\right)\ps{2}  S(v\ns{1}\eta\ns{1}f^q\ps{2})\ps{2}.
\end{align*}
Next, we observe the invertibility by introducing
\begin{align*}
& \Ic^{-1}(v\ot \eta \ot f^0 \odots f^q) := v\ot \eta \ot f^0\ps{1} \ot f^0\ps{2}f^1\ps{1} \odots f^0\ps{q}\ldots f^{q-2}\ps{2}f^{q-1}\ve(f^q).
\end{align*}
Indeed,
\begin{align*}
& \Ic(\Ic^{-1}(v\ot \eta \ot f^0 \odots f^q)) = \\
& \Ic(v\ot \eta \ot f^0\ps{1} \ot f^0\ps{2}f^1\ps{1} \odots f^0\ps{q}\ldots f^{q-2}\ps{2}f^{q-1}\ve(f^q)) = \\
&v\ns{0} \ot \eta\ns{0} \ot f^0\ps{1}\ps{1} \ot S(f^0\ps{1}\ps{2})f^0\ps{2}\ps{1}f^1\ps{1}\ps{1} \ot\cdots \\
&\cdots\ot S(f^0\ps{q-1}\ps{2}\ldots f^{q-2}\ps{1}\ps{2})f^0\ps{q}\ps{1}\ldots f^{q-2}\ps{2}\ps{1}f^{q-1}\ps{1} \ot \\
& \hspace{2cm} S(v\ns{1}\eta\ns{1}f^0\ps{q}\ps{2}\ldots f^{q-2}\ps{2}\ps{2}f^{q-1}\ps{2}))\ve(f^q) = \\
& v\ns{0} \ot \eta\ns{0} \ot f^0\ps{1} \odots f^{q-1}\ps{1} \ot \ve(f^q)S(f^0\ps{2}\ldots f^{q-2}\ps{2}f^{q-1}\ps{2})S(v\ns{1} \eta\ns{1}) = \\
& v \ot \eta \ot f^0\ps{1}\odots f^{q-1}\ps{1} \ot \ve(f^q\ps{1})S(f^0\ps{2}\ldots f^{q-2}\ps{2}f^{q-1}\ps{2}) f^0\ps{3}\ldots f^{q-1}\ps{3}f^q\ps{2} = \\
& v\ot \eta \ot f^0 \odots f^q,
\end{align*}
where we used the coinvariance condition in the fourth equality. Similarly we may observe
\[
\Ic^{-1}( \Ic(v\ot \eta \ot f^1 \odots f^q)) = v\ot \eta \ot f^1 \odots f^q.
\]
\end{proof}
\ni As a result, by the transfer of structure, we obtain
\begin{align*}
& d^{\rm coinv}_{\rm CE}:=\Ic \circ d_{\rm CE}\circ\Ic^{-1}: C^{p,q}_{\rm coinv}(\Om_{n\d}^{\leq1},\Fs^\ast,\Fc(N)) \lra C^{p+1,q}_{\rm coinv}(\Om_{n\d}^{\leq1},\Fs^\ast,\Fc(N)) \\
& b^{\ast\,{\rm coinv}}_N:=\Ic \circ b^\ast_N\circ\Ic^{-1}: C^{p,q}_{\rm coinv}(\Om_{n\d}^{\leq1},\Fs^\ast,\Fc(N)) \lra C^{p,q+1}_{\rm coinv}(\Om_{n\d}^{\leq1},\Fs^\ast,\Fc(N))
\end{align*}
Now we identify, just as \cite[Prop. 1.16]{MoscRang11} the coinvariant bicomplex \eqref{coinv-bicomplex} with the bicomplex \eqref{group-bicomplex}.

\begin{proposition}
The map 
\[
\Jc:C^{\ast,\ast}_{\rm coinv}(\Om_{n\d}^{\leq1},\Fs^\ast,\Fc(N)) \to C_{\rm pol}^\ast(N,\Om_1^{\leq1}\ot \wg^\ast\,\Fs^\ast),
\]
given by
\[
\Jc(v\ot \eta \ot f^0\odots f^q)(\psi_0,\ldots, \psi_q) = v \ot \eta \,f^0(\psi_0)\ldots f^q(\psi_q),
\]
is an isomorphism of bicomplexes.
\end{proposition}

\begin{proof}
We begin with the well-definedness. Indeed,
\begin{align*}
& \Jc(v\ot \eta \ot f^0\odots f^q)(\psi_0\psi,\ldots, \psi_q\psi) = \\
& \Jc(v\ot \eta \ot f^0\ps{1}\odots f^q\ps{1} \ot f^0\ps{2}\ldots f^q\ps{2})(\psi_0,\ldots, \psi) = \\
& \Jc(v\ns{0} \ot \eta\ns{0} \ot f^0\odots f^q \ot S(v\ns{1}\eta\ns{1}))(\psi_0,\ldots, \psi_q,\psi) = \\
& v\ns{0} \ot \eta\ns{0}\, f^0(\psi_0)\ldots f^q(\psi_q)S(v\ns{1}\eta\ns{1})(\psi) = \\
& v\ns{0} \ot \eta\ns{0}\, f^0(\psi_0)\ldots f^q(\psi_q)(v\ns{1}\eta\ns{1})(\psi^{-1})= \\
& v\ns{0} \ot \eta\ns{0}\, f^0(\psi_0)\ldots f^q(\psi_q)v\ns{1}(\psi^{-1})\eta\ns{1}(\psi^{-1}) = \\
& \psi^{-1}\cdot (v \ot \eta)\, f^0(\psi_0)\ldots f^q(\psi_q) = \left((v \ot \eta)\, f^0(\psi_0)\ldots f^q(\psi_q)\right) \cdot \psi= \\
& \Jc(v\ot \eta \ot f^0\odots f^q)(\psi_0,\ldots, \psi_q) \cdot \psi.
\end{align*}
Let us now check the compatibility with the coboundaries. To begin with, we have
\begin{align*}
& b_N(\Jc(v\ot \eta \ot f^0\odots f^q))(\psi_0,\ldots, \psi_{q+1}) =\\
& \sum_{i=0}^{q+1}\,(-1)^i\,\Jc(v\ot \eta \ot f^0\odots f^q)(\psi_0,\ldots, \widehat{\psi}_i,\ldots,\psi_{q+1}) =\\
& \sum_{i=0}^{q+1}\,(-1)^i\, v\ot \eta \,f^0(\psi_0)\ldots f^i(\psi_{i+1})\ldots  f^q(\psi_{q+1}).
\end{align*}
On the other hand,
\begin{align*}
& b^{\ast\,{\rm coinv}}_N(v\ot \eta \ot f^0\odots f^q) =  \Ic \circ b^\ast_N\circ\Ic^{-1} (v\ot \eta \ot f^0\odots f^q) = \\
& \Ic \circ b^\ast_N \left(v \ot \eta \ot f^0\ps{1} \ot f^0\ps{2}f^1\ps{1} \odots f^0\ps{q}\ldots f^{q-2}\ps{2}f^{q-1}\ve(f^q)\right) = \\
& \Ic \Big( v \ot \eta \ot 1 \ot f^0\ps{1} \ot f^0\ps{2}f^1\ps{1} \odots f^0\ps{q}\ldots f^{q-2}\ps{2}f^{q-1}\ve(f^q) + \\
& \sum_{i=0}^q\,(-1)^i\,v \ot \eta \ot f^0\ps{1} \ot f^0\ps{2}f^1\ps{1} \odots \D(f^0\ps{i}f^1\ps{i-1}\ldots f^{i-1}\ps{1}) \ot \cdots \\
& \hspace{2cm} \cdots\ot f^0\ps{q}\ldots f^{q-2}\ps{2}f^{q-1}\ve(f^q) + \\
& (-1)^{q+1} \, v\ns{0} \ot \eta\ns{0} \ot f^0\ps{1} \ot f^0\ps{2}f^1\ps{1} \ot\cdots \\
& \hspace{2cm} \cdots\ot f^0\ps{q}\ldots f^{q-2}\ps{2}f^{q-1}\ve(f^q) \ot S(v\ns{1})\eta\ns{-1} \Big).
\end{align*}
Now we note, in view of the coinvariance property, that
\begin{align*}
& \Ic\Big( v \ot \eta \ot 1 \ot f^0\ps{1} \ot f^0\ps{2}f^1\ps{1} \odots f^0\ps{q}\ldots f^{q-2}\ps{2}f^{q-1}\ve(f^q) \Big) = \\
& v\ns{0} \ot \eta\ns{0} \ot 1 \ot S(1)f^0\ps{1}\ps{1} \ot S(f^0\ps{1}\ps{2})f^0\ps{2}\ps{1}f^1\ps{1}\ps{1} \ot\cdots \\
& \cdots\ot S(f^0\ps{q-1}\ps{2}\ldots f^{q-2}\ps{1}\ps{2}) f^0\ps{q}\ps{1}\ldots f^{q-2}\ps{2}\ps{1}f^{q-1}\ps{1}\ve(f^q) \ot \\
& S(v\ns{1}\eta\ns{1}f^0\ps{q}\ps{2}\ldots f^{q-2}\ps{1}\ps{2}f^{q-1}\ps{2}) = \\
& v \ot \eta \ot 1 \ot f^0 \odots f^q,
\end{align*}
that
\begin{align*}
& \Ic\Big( v \ot \eta \ot f^0\ps{1} \ot f^0\ps{2}f^1\ps{1} \odots \D(f^0\ps{i}f^1\ps{i-1}\ldots f^{i-1}\ps{1}) \ot \cdots \\
& \hspace{2cm} \cdots\ot f^0\ps{q}\ldots f^{q-2}\ps{2}f^{q-1}\ve(f^q)\Big) = \\
& v\ns{0} \ot \eta\ns{0} \ot f^0\ps{1}\ps{1} \ot S(f^0\ps{1}\ps{2})f^0\ps{2}\ps{1}f^1\ps{1}\ps{1} \ot \cdots \\
& \cdots \ot S(f^0\ps{i}\ps{1}\ps{2}f^1\ps{i-1}\ps{1}\ps{2}\ldots f^{i-1}\ps{1}\ps{1}\ps{2})f^0\ps{i}\ps{2}\ps{1}f^1\ps{i-1}\ps{2}\ps{1}\ldots f^{i-1}\ps{1}\ps{2}\ps{1} \ot \cdots \\
& \cdots \ot S(v\ns{1}\eta\ns{1}f^0\ps{q}\ps{2}\ldots f^{q-2}\ps{2}\ps{2}f^{q-1}\ps{2}) = \\
& v \ot \eta \ot f^0 \odots f^{i-1} \ot 1 \ot f^i \odots f^q,
\end{align*}
and that
\begin{align*}
& \Ic\Big( v\ns{0} \ot \eta\ns{0} \ot f^0\ps{1} \ot f^0\ps{2}f^1\ps{1} \ot\cdots \\
& \hspace{2cm} \cdots\ot f^0\ps{q}\ldots f^{q-2}\ps{2}f^{q-1}\ve(f^q) \ot S(v\ns{1})\eta\ns{-1} \Big) = \\
& v\ns{0}\ns{0} \ot \eta\ns{0}\ns{0} \ot f^0\ps{1}\ps{1} \ot S(f^0\ps{1}\ps{2})f^0\ps{2}\ps{1}f^1\ps{1}\ps{1} \ot\cdots \\
&\hspace{1.5cm} \cdots\ot S(f^0\ps{q}\ps{2}\ldots f^{q-2}\ps{2}\ps{2}f^{q-1}\ps{2}\ve(f^q))S(v\ns{1}\eta\ns{1})\ps{1} \ot \\
& S(v\ns{0}\ns{1}\eta\ns{0}\ns{1}S(v\ns{1}\eta\ns{1})\ps{2}) = \\
& v\ot \eta \ot f^0 \odots f^q \ot 1.
\end{align*}
As a result,
\begin{align*}
& \Jc(b^{\ast\,{\rm coinv}}_N(v\ot \eta \ot f^0\odots f^q))(\psi_0, \ldots, \psi_{q+1}) = \sum_{i=0}^{q+1}\,(-1)^i\, v\ot \eta \,f^0(\psi_0)\ldots f^i(\psi_{i+1})\ldots  f^q(\psi_{q+1}).
\end{align*}
We thus observe that $\Jc \circ b^{\ast\,{\rm coinv}}_N = b_N \circ \Jc$. Let us proceed to $b_\Fs \circ \Jc = \Jc \circ d_{\rm CE}^{\rm coinv}$, that is, $b_\Fs \circ \Jc = \Jc \circ \Ic \circ d_{\rm CE}\circ \Ic^{-1}$, the compatibility with the vertical coboundary maps. To this end, we shall verify
\[
\Ic^{-1} \circ \Jc^{-1} \circ b_\Fs \circ \Jc = d_{\rm CE}\circ \Ic^{-1}.
\]
On the one hand we have
\begin{align*}
& d_{\rm CE}\circ \Ic^{-1} (v \ot \eta \ot f^0\odots f^q) = \\
& d_{\rm CE} \left(v \ot \eta \ot f^0\ps{1} \ot f^0\ps{2}f^1\ps{1} \odots f^0\ps{q}\ldots f^{q-2}\ps{2}f^{q-1}\ve(f^q)\right) = \\
& d_{\rm CE}^\Om(v \ot \eta) \ot f^0\ps{1} \ot f^0\ps{2}f^1\ps{1} \odots f^0\ps{q}\ldots f^{q-2}\ps{2}f^{q-1}\ve(f^q) - \\
& v \ot \t^j\wg \eta \ot e_j\bullet\left(f^0\ps{1} \ot f^0\ps{2}f^1\ps{1} \odots f^0\ps{q}\ldots f^{q-2}\ps{2}f^{q-1}\ve(f^q)\right),
\end{align*}
and on the other hand
\begin{align*}
& \Ic^{-1} \circ \Jc^{-1} \circ b_\Fs \circ \Jc (v \ot \eta \ot f^0\odots f^q) = \\
& \Ic^{-1} \left(d_{\rm CE}^\Om(v \ot \eta) \ot f^0 \odots f^q\right) - \\
& \hspace{2cm} \Ic^{-1} \left(v \ot \t^j \wg \eta \ot e_j\rt (f^0\odots f^q)\right) = \\
& d_{\rm CE}^\Om(v \ot \eta) \ot f^0\ps{1} \ot f^0\ps{2}f^1\ps{1} \odots f^0\ps{q}\ldots f^{q-2}\ps{2}f^{q-1}\ve(f^q) - \\
& v \ot \t^j\wg \eta \ot e_j\bullet\left(f^0\ps{1} \ot f^0\ps{2}f^1\ps{1} \odots f^0\ps{q}\ldots f^{q-2}\ps{2}f^{q-1}\ve(f^q)\right),
\end{align*}
where the last equality is a direct consequence of \cite[Eq. (1.50)]{MoscRang11}.
\end{proof}

\ni As a result, we now have the classes 
\begin{equation*}
\Jc \circ \Ic \left( [\one \ot \t^0]_1 \right) \in E_1^{0,1}, \qquad \Jc \circ \Ic \left( [\sum_{i\geq 1}\,(i+1)if^{i-1} \ot {\bf x}_i]_1 \right) \in E^{1,0}_1
\end{equation*}
in the bicomplex \eqref{group-bicomplex}. We note further that, since they are not coboundaries of continuous 0-cochains, they survive in the continuous bicomplex \eqref{cont-group-bicomplex} which is the $E_1$-term of the Cartan-Leray spectral sequence computing the group cohomology $H_{\rm cont}^\ast(\Diff(\Rb),\Om_1^{\leq1})$.

\begin{proposition}
The cochain
\[
d\ell \in C_{\rm cont}^1(N,\Om_1^{\leq1}\ot \wg^0\Fs^\ast),
\]
given by
\[
d\ell: \psi \mapsto d\log(\psi'(x))dx = \frac{\psi''(x)}{\psi'(x)}\,dx,
\]
is a cocycle in the bicomplex \eqref{cont-group-bicomplex}.
\end{proposition}

\begin{proof}
Let us first observe that
\begin{align*}
& b_N(d\ell)(\psi_1,\psi_2) = (d\ell)(\psi_2) - (d\ell)(\psi_1\psi_2) + (d\ell)(\psi_1)\cdot \psi_2 = \\
& \frac{\psi_2''(x)}{\psi_2'(x)}\,dx - \frac{(\psi_1\psi_2)''(x)}{(\psi_1\psi_2)'(x)}\,dx + \frac{\psi_1''(x)}{\psi_1'(x)}\,dx\cdot \psi_2 = \\
& \frac{\psi_2''(x)}{\psi_2'(x)}\,dx - \left[\frac{\psi_1''(\psi_2(x))\psi_2'(x)^2 + \psi_1'(\psi_2(x))\psi_2''(x)}{\psi_1'(\psi_2(x))\psi_2'(x)}\,dx\right] + \frac{\psi_1''(\psi_2(x))}{\psi_1'(\psi_2(x))}\psi_2'(x)\,dx = 0.
\end{align*}
Next, we see that
\begin{align*}
& b_\Fs(d\ell)(\psi) = d^\Om_{\rm CE}(d\ell(\psi)) - \t^{-1} \wg (e_{-1} \rt d\ell)(\psi) - \t^0 \wg (e_0 \rt d\ell)(\psi) = \\
& d\ell(\psi) \cdot e_{-1} \ot \t^{-1} + d\ell(\psi)\cdot e_0 \ot \t^0 - d\ell(\psi\lt e_{-1}) \ot \t^{-1} - d\ell(\psi\lt e_0) \ot \t^0 = \\
& \left(d\ell(\psi) \cdot e_{-1} - d\ell(\psi\lt e_{-1})\right)\ot \t^{-1} + \left(d\ell(\psi) \cdot e_0 - d\ell(\psi\lt e_0)\right)\ot \t^0.
\end{align*}
We thus have to recall that
\[
\exp(te_{-1}):\Rb \lra \Rb, \qquad \exp(te_{-1})(x) = x + t,
\]
and
\[
\exp(te_0):\Rb \lra \Rb, \qquad \exp(te_0)(x) = tx,
\]
and on the other hand, for any $\phi \in \Diff(\Rb)$,
\[
\phi = \vp\psi, \qquad \vp(x) = \phi'(0)x + \phi(0), \qquad \psi = \vp^{-1}\phi.
\]
Then since the mutual actions satisfy $\psi\vp = (\psi \rt \vp)(\psi \lt \vp)$, we have
\[
(\psi \lt \vp)(x) = \frac{(\psi\vp)(x)}{(\psi\vp)'(0)} - \frac{(\psi\vp)(0)}{(\psi\vp)'(0)}.
\]
In particular,
\[
(\psi \lt \exp(te_{-1}))(x) = \frac{\psi(x+t)}{\psi'(t)} - \frac{\psi(t)}{\psi'(t)},
\]
and keeping $\psi(0)=0$ and $\psi'(0) = 1$ in mind,
\[
(\psi \lt \exp(te_0))(x) = \frac{\psi(tx)}{t}.
\]
Hence, we have
\[
d\ell(\psi \lt \exp(te_{-1}))(x) = \frac{\psi''(x+t)}{\psi'(x+t)},
\]
and 
\[
d\ell(\psi \lt \exp(te_0))(x) = \frac{t\psi''(tx)}{\psi'(tx)}.
\]
As a result,
\begin{align*}
& b_\Fs(d\ell)(\psi) = \\
& \dt\, \left[\frac{\psi''(x+t)}{\psi'(x+t)}(x+t)' - \frac{\psi''(x+t)}{\psi'(x+t)}\right] \ot \t^{-1} +  \dt\, \left[\frac{\psi''(tx)}{\psi'(tx)}(tx)' - \frac{t\psi''(tx)}{\psi'(tx)}\right] \ot \t^0 = 0.
\end{align*}
\end{proof}

\begin{proposition}
The cochain
\[
\one \ot \t^0 + \ell \in C_{\rm cont}^0(N,\Om_1^{\leq1}\ot \wg^1\Fs^\ast) \oplus C_{\rm cont}^1(N,\Om_1^{\leq1}\ot \wg^0\Fs^\ast),
\]
where
\[
\ell: \psi \mapsto \log(\psi'(x)),
\]
is a cocycle in the bicomplex \eqref{cont-group-bicomplex}.
\end{proposition}

\begin{proof}
To begin with, we already have
\[
b_\Fs(\one \ot \t^0) = \one\cdot e_j \ot \t^j\wg \t^0 + \one \ot d_{\rm DR}(\t^0) = 0.
\]
On the other hand, for the horizontal coboundary we observe from \eqref{F-coact-s}, and \cite[Eqn. (3.33)]{Antal-thesis} that
\begin{align*}
& b_N(\one \ot \t^0) (\psi) = (\one \ot \t^0)\cdot \psi = \one \ot \t^0\ns{-1}(\psi)\, \t^0\ns{0} = \\
& \one \ot \d_1(\psi)\t^{-1} = - \one \ot \psi''(0)\t^{-1}.
\end{align*}
We proceed to $\ell \in C_{\rm cont}^1(N,\Om_1^{\leq1}\ot \wg^0\Fs^\ast)$. On one hand we have
\begin{align*}
& b_\Fs(\ell)(\psi) = d^\Om_{\rm CE}(\ell(\psi)) - \t^{-1} \wg (e_{-1} \rt \ell)(\psi) - \t^0 \wg (e_0 \rt \ell)(\psi) = \\
& \left(\ell(\psi) \cdot e_{-1} - \ell(\psi\lt e_{-1})\right)\ot \t^{-1} + \left(\ell(\psi) \cdot e_0 - \ell(\psi\lt e_0)\right)\ot \t^0 = \\
& \dt\, \left[\log(\psi'(x+t))(x+t)' - \log\left(\frac{\psi'(x+t)}{\psi'(t)}\right)\right] \ot \t^{-1} + \\
& \hspace{2cm} \dt\, \left[\log(\psi'(tx))(tx)' - \log(\psi'(tx))\right] \ot \t^0 = \psi''(0)\t^{-1},
\end{align*}
and on the other hand,
\begin{align*}
& b_N(\ell)(\psi_1,\psi_2) = \ell(\psi_2) - \ell(\psi_1\psi_2) + \ell(\psi_1)\cdot \psi_2 = \\
& \log(\psi'_2(x)) - \log(\psi'_1(\psi_2(x))\psi'_2(x)) + \log(\psi'_1(\psi_2(x))) = 0.
\end{align*}
\end{proof}

\ni Since there are the only two classes in $H^1(\Diff(\Rb),\Om_1^{\leq1})$, by Proposition \ref{prop-van-est-diff}, we conclude
\[
[\Jc \circ \Ic \left( \lambda' \right)]_1 = [\one \ot \t^0]_1 + [\ell]_1, \qquad [\Jc \circ \Ic \left( \mu' \right)] = [d\ell]_1.
\]

\bibliographystyle{amsplain}
\bibliography{Rangipour-Sutlu-References}{}

\end{document}